\pdfoutput=1
\documentclass[11pt]{amsart}
\usepackage{geometry}                
\usepackage[parfill]{parskip}   
\usepackage{graphicx,color,amsthm}
\usepackage{amssymb}
\usepackage{epstopdf}
\usepackage{pdfsync}
\usepackage{caption}
\usepackage{subcaption}
\usepackage{mathtools}
\usepackage{fullpage}
\usepackage{tabularx}
\theoremstyle{plain} 
\newtheorem{thm}{Theorem}

\newtheorem{cor}[thm]{Corollary} 
\newtheorem{lem}[thm]{Lemma}

\theoremstyle{definition} 
\newtheorem{defn}{Definition}
\newtheorem{ex}{Example}

\title{Linking in Cyclic Branched Covers\\ and Satellite (non)-Homomorphisms}
\author{Patricia Cahn and Alexandra Kjuchukova }
\date{January 2023}

\begin{document}
\raggedbottom
\maketitle
\begin{abstract}
    Let $K\subset S^3$ be a knot and $\eta, \gamma \subset S^3\backslash K$ be simple closed curves.    Denote by $\Sigma_q(K)$ the $q$-fold cyclic branched cover of $K$. We give an explicit formula for computing the linking numbers between lifts of $\eta$ and $\gamma$ to $\Sigma_q(K)$. 
    As an application, we evaluate, in a variety of cases, an obstruction to satellite operations inducing homomorphisms on smooth concordance.  
   
\end{abstract}

\section{Introduction}

Given two rationally null-homologous, disjoint curves in an oriented 3-manifold, their linking number $\text{lk}(\alpha,\beta)\in \mathbb{Q}$ is well-defined and symmetric \cite{birman1980seifert}.  Linking numbers in branched covers of $S^3$ were first computed by Reidemeister~\cite{reidemeister2013knotentheorie} and were one of the early methods for distinguishing knots. In the 1960s, Ken Perko wrote a computer program for computing linking numbers between the branch curves in 3-fold irregular branched covers~\cite{perko1964invariant} and used it to help complete the classification of knots through 10 crossings. The authors extended this technique to computing linking numbers, in 3-fold irregular covers, between ``pseudo-branch curves", that is, connected components of lifts to the branch cover of curves which lie in the complement of the branching set~\cite{cahn2021linking}. Perko's computation of linking numbers between branch curves was also generalized to all odd-fold irregular dihedral covers in~\cite{cahn2021dihedral}. Linking numbers in branched covers can be used~\cite{cahn2018computing} to compute the Rokhlin $\mu$ invariant~\cite{cappell1975invariants} and Kjuchukova's $\Xi_p$ invariant~\cite{kjuchukova2018dihedral, geske2021signatures}, among other applications. 

In this note, we give a formula for computing linking numbers between pseudo-branch curves in {\it cyclic} branched covers. As an application, we show, using an obstruction from~\cite{lidman2022linking}, that certain satellite patterns do not induce homomorphisms on smooth concordance. We thus verify a conjecture of Hedden~\cite{heddenconjecture} for these patterns. 
Linking numbers of the types considered here can also detect satellite patterns of infinite rank, using~\cite{hedden2021satellites}.

There exist alternative methods for computing linking numbers in cyclic branched covers. Under the assumption that the linking numbers between the branch set and (projections of) the pseudo-branch curves are zero, formulas for the linking numbers are given in~\cite{przytycki2004linking}. Additionally, a 
formula for linking numbers in a rational homology 3-sphere described by a Kirby diagram 
is provided in~\cite{cha2002signatures}. 
A method for passing between a diagram of the branching set and a Dehn surgery presentation of its cyclic cover is given in~\cite{akbulut1980branched}; however, this gets gnarly as the degree of the cover or the Seifert genus of the branching set increases. In contrast, the input for our computation is a diagram of the branching set $K$ (in $S^3$) and the curves of interest in the complement of $K$. The computation is not sensitive to the Seifert genus of $K$ -- see Example~\ref{ex:gnarly-slice} in Section~\ref{sec:apps} -- and makes no assumption about the linking numbers between the branch and pseudo-branch curves in $S^3$. We have implemented our computation in \cite{cahn2023cyclic}, so the linking numbers can be calculated genuinely easily, even for high degree covers. 

We include an abridged version of our result below, suppressing a lot of the notation used in the linking number formula. The full statement is deferred to Section ~\ref{general.sec}. 
For ease of exposition, we first prove the theorem in Section~\ref{sec:thms}, under an additional assumption on the linking numbers, in $S^3$, of the pseudo-branch curves and the branch curve, as needed for our main application~\cite{lidman2022linking}. Note that we do not require that the cover be a rational homology sphere.

\begin{thm}[Main Theorem, Abridged] Let $K\subset S^3$ be a knot.  Denote by $\Sigma_q(K)$ the $q$-fold branched cover of $K$. We first determine whether each closed, connected component of the lifts of $\gamma$ and $\eta$ is rationally null-homologous in $\Sigma_q(K).$ When this is the case, the linking numbers between connected components of the lifts of $\gamma$ and $\eta$ in $\Sigma_q(K)$ can be computed by the formula in Equation~\ref{linking-formula-general}. The input needed for performing the computation is contained in the Dowker-Thistlethwaite code for a diagram of $K\cup\gamma\cup\eta$.    
\end{thm}

The notation needed to properly state (or apply) our theorem is introduced in Section~\ref{sec:cellular}, where we also describe the geometric construction used in the proof. 
Step by step instructions for how to use the computational tool provided in \cite{cahn2023cyclic} are given in the appendix.

\section{Cell structure on the cyclic branched cover of a knot}\label{sec:cellular}
We endow $S^3$ with a cell structure determined by the link $K\cup\eta$. Recall that $K$ denotes the branching set and $\eta\subset S^3\backslash K$ is a pseudo-branch curve. We then lift this cell structure to the branched cover $\Sigma_q(K)$. In this section, we describe this structure and introduce notation for the cells. Analogous cellular decompositions of the 3-sphere and its irregular dihedral covers are used in~\cite{perko1964invariant, kjuchukova2018dihedral, cahn2018computing, cahn2021linking}.

\begin{figure}
    \centering
    \includegraphics[width=\textwidth]{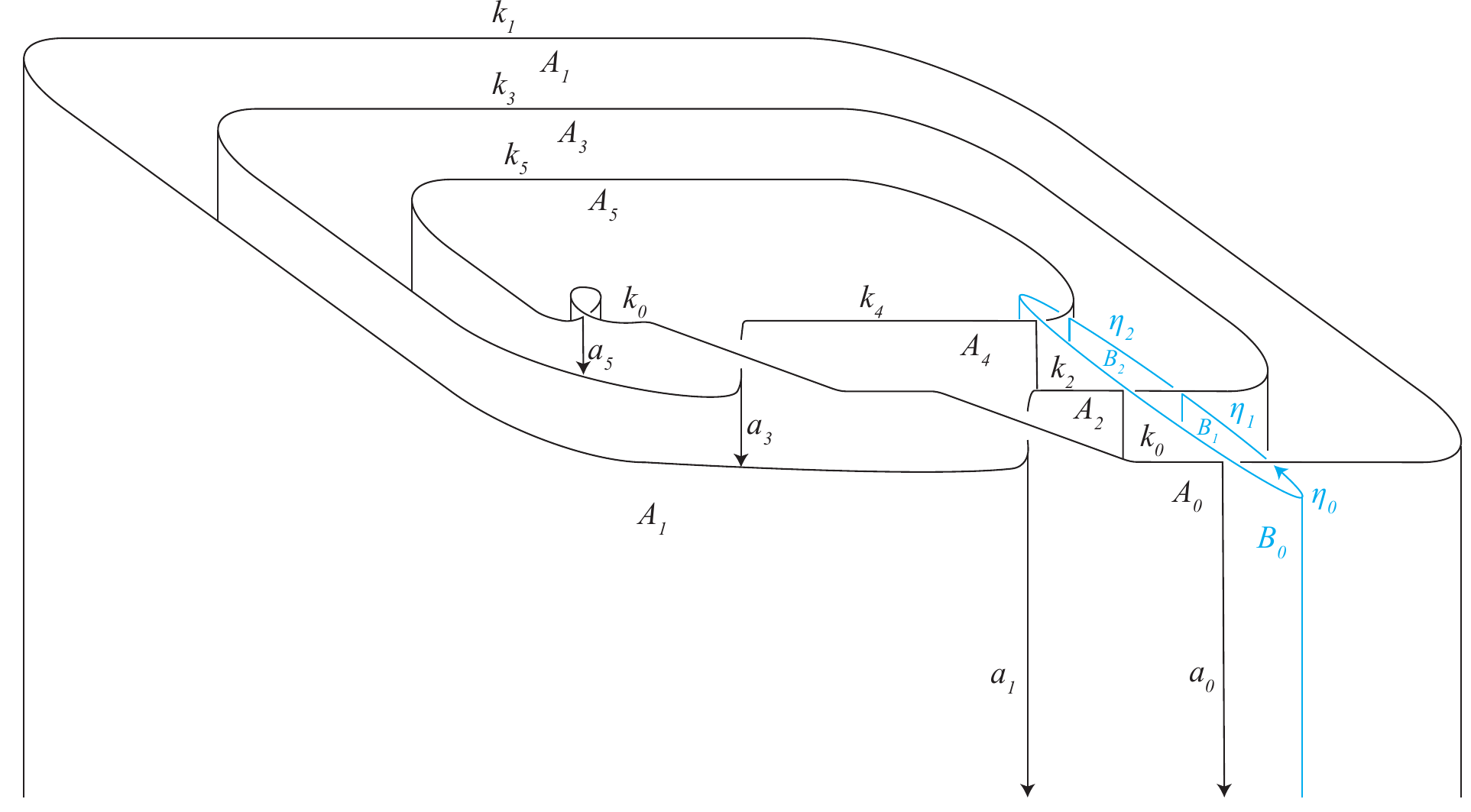}
    \caption{A cell structure on $S^3$ determined by the $(3,1)$-cabling pattern, represented by a 2-component link $K\cup \eta$.}
    \label{fig:31cablecells}
\end{figure}

Throughout this section, the reader should refer to Figure ~\ref{fig:31cablecells}.  

Let $D$ be a diagram of $K\cup\eta$, regarded as a projection with double points. Consider  $c(D)$, the cone on the diagram, illustrated in Figure ~\ref{fig:31cablecells} and described in terms of a cell structure in the following paragraphs. We describe a cell structure on $S^3$ in terms of $D$. The 0-cells are the cone point $c$ and the double points $c_j$ in $D$. (Each $c_j$ is regarded as belonging to the understrand at this crossing.) Adopting the language used by Perko, we categorize the 1-cells as either {\it horizontal}, corresponding to strands in $D$, and {\it vertical}, which are the cones on the $c_j$. The 2-cells, called ``walls", are the cones on the horizontal 1-cells, and they are indexed by the strands $d_i$ of $D$. The boundary of each 2-cell $C^2$ consists of: a single horizontal 1-cell $d_i$ (where $C^2_i=c(d_i)$); two vertical 1-cells, the cones on the endpoints of $d_i$; additional pairs of vertical 1-cells, called ``slits", corresponding to crossings where $d_i$ is the overstrand\footnote{The ``slits" are part of the attaching map of the 2-cell. They come in canceling pairs so do not contribute to the 1-chain $\partial C^2_i$. However, the presence of slits does sometimes contribute to the 1-chains bounded by lifts of $C^2_i$.}.  The 2-skeleton of the cell structure on $S^3$ is illustrated in Figure~\ref{fig:31cablecells}, with the cone point regarded as a point at infinity. There is a single 3-cell, the complement of the cone. 

We proceed to name the strands of $D$, the 1-cells and the 2-cells. We orient $K$ and fix a 0-cell which we label $c_0$. The remaining 0-cells on $K$ get consecutive indices, proceeding in the direction of the orientation. The horizontal 1-cells which are strands of $K$ are then labeled $k_0, k_1, \dots, k_{n-1}$, with $\partial(k_i)=c_{i+1}-c_i$. A remark on the number of strands: in order to simplify labeling the lifts of two-cells later, we henceforth assume that the writhe of $K$ is 0~mod~$q$, after performing some Reidemeister-I moves if necessary. We also orient $\eta$ and number the horizontal 1-cells which are strands of $\eta$ in the analogous manner: $\eta_0, \eta_1, \dots, \eta_{m-1}$. If $c_i$ is a point on $K$, the vertical 1-cell whose boundary is $c-c_i$ is labeled $a_i$.  If $c_i$ is a point on $\eta$, the vertical 1-cell whose boundary is $c-c_i$ is labeled $b_i$.  The 2-cell whose horizontal boundary is $k_i$ will be denoted by $A_i$. The 2-cell whose horizontal boundary is $\eta_i$ will be denoted by $B_i$. We orient each 2-cell so that $\partial(A_i)= k_i + a_{i+1}-a_i$ and $\partial(B_i)= \eta_i + b_{i+1}-b_i$. There's a single 3-cell (whose interior is the complement of the cone) denoted $E$. 

We now describe the induced cell structure on $\Sigma_q(K)$. Each strand $k_i$ has a single lift, which we also denote by $k_i$. Every 0-, 1-, 2- and 3-cell whose interior is contained in $S^3\backslash K$ has $q$ lifts, labeled in the natural way: the lifts of $a_i$ are $a_i^1,\dots a_i^q$; of $A_i$ are $A_i^1,\dots A_i^q$; analogously for the $\eta_i$ and $B_i$; the 3-cells in $\Sigma_q(K)$ are $E^1, \dots, E^q$. Some of the notation can be chosen freely, and the rest is determined by the given diagram of $K\cup\eta.$ We assign superscripts in such a way that $\partial(A_i^j)= k_i^j + a_{i+1}^j-a_i^j +\dots$ and $\partial(B_i^j)= \eta_i^j +\dots$.  (It is to be understood that these 1-cells do not appear again in the rest of the chain, that is, the coefficient of $k_i^j$ in $\partial(A_i^j)$ is precisely 1, etc..)  Remark also that the 1-cell $k_i$ belongs to the boundary of $A_i^j$ for all $j$. For a fixed $i$, we let the superscripts of the cells $A_i^j$ increase in the clockwise direction (if the orientation of the knot points into the page). Put differently, if an Ambler\footnote{A convenient visualization device introduced, under another name, by Perko~\cite{perko1964invariant}. } walking around  $\Sigma_q(K)$ stands on $A_i^j$, with the cell $a_i^j$ to his left, and faces in the direction of the orientation of $K$, then the 2-cell to his left is $A_i^{j-1}$.  See Figure~\ref{fig:roladex}. Now let the Ambler continue along a push-off of $K$ until he crosses a wall (that is, a 2-cell whose horizontal boundary is a lift of the overstrand at crossing $c_i$). After crossing the wall, the strand to the left of the Ambler is $k_{i+1}$ and the 2-cell he's walking on is an $A_{i+1}^{s}$ for some value of $s$. We pick our convention so that $s=j$, that is, the superscripts remain unchanged when passing through a wall. See Figure ~\ref{fig:knot_over_knot}.  However, each time a wall $A_i^j$ is crossed, the ambient 3-cell changes, since the meridian of the horizontal boundary of $A_i$ permutes the 3-cells. (Passing through a wall $B_i^j$, whose horizontal boundary is a lift of $\eta$, does not change the ambient 3-cell.) 
 
 In summary, each 2-cell in $S^3$ labeled $A_i$ has $q$ lifts, $A_i^1,\dots A_i^q$, to $\Sigma_q(K)$. These 2-cells all share a boundary arc, $k_i$. For a fixed $i$, the superscripts of the cells $A_i^j$ increase when going along a positively oriented meridian of the arc $k_i$, see Figure~\ref{fig:roladex}.  Furthermore, standing on  $A_i^j$ and moving along a push-off of $k_i$, after passing through a 2-cell, the Ambler is on $A_{i+1}^j$. Some labels of 2-cells are illustrated in Figure~\ref{fig:knot_over_knot}.
 
 Next, we number the 3-cells of $\Sigma_q(K)$. Suppose the Ambler stands on a push-off of $K$ into $A_0^j$, with the 1-cell $K_0$ to his left and facing in the direction of the orientation of $K$. We label the 3-cell his head is by $E^j$. In addition, we assume that the subscript denoting the 3-cell increases clockwise when passing through a 2-cell which bounds $K_0$, see Figure~\ref{fig:roladex}. This choice, made around the 1-cell $K_0$, in fact determines the positions of all 3-cells with respect to the entire 2-skeleton. As the Ambler moves along the push-off of $K$, he will pass through walls, or vertical 2-cells, whenever a strand $K_r$ turns into a strand $K_{r+1}$, that is, at all lifts of crossings where an arc of $K$ is the understrand.  Passing through a wall at a lift of a positive crossing increases the superscript on the 3-cell by 1 (since a positively oriented meridian of the branching set acts on the 3-cells by the $q$-cycle $(1 2 3\dots q)$). Similarly, passing through a wall at a lift of a negative crossing decreases the superscript of the 3-cell by 1. We introduce a function, $\omega_K(i)(j)$, to keep track of the positions of 3-cells: when the Ambler has the 1-cell $k_i$ to his left, and he is standing on the 2-cell $A_i^j$, his head will be in the 3-cell $E^{\omega_K(i)(j)}$. 
  Lastly, recall that we have arranged that the writhe of the knot is 0~mod~$q$. Therefore, if the Ambler starts a walk on a right push-off of $k_0$, standing on $A_0^j$, after completing a full circle along a push-off of $K$, he has his head back in $E^j$, the 3-cell he started in. 
 \begin{figure}
     \centering
     \includegraphics[width=5.5in]{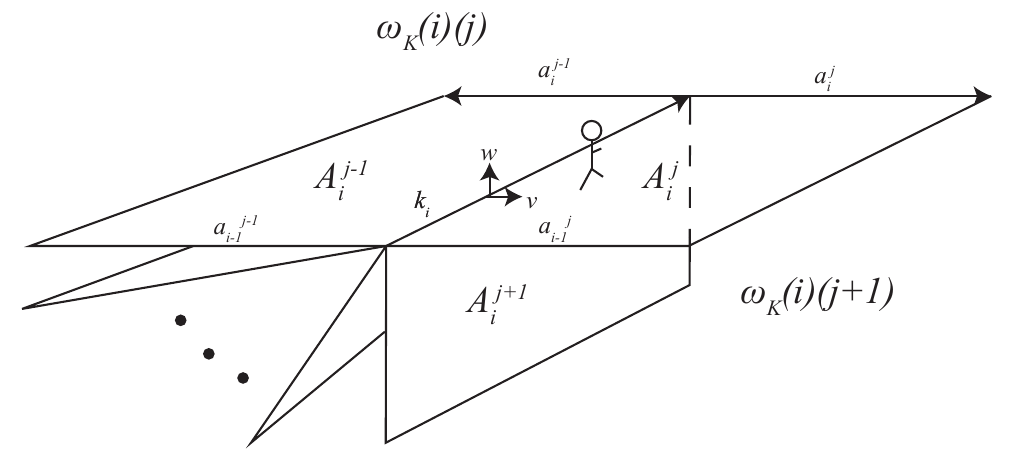}
     \caption{Roladex of lifts of the cell $A_i$. The Ambler is on $A_i^j$ with $k_i$ to his left. His head is in the 3-cell $E^{\omega_K(i)(j)}$. Throughout, we assume that the Ambler is oriented positively with respect to $A_i^j$. That is, thee vectors $v$, pointing to the Ambler's right, and $w$, pointing from the Ambler's feet to head, together with the tangent vector to $k_i$, form a positive frame.  }
     \label{fig:roladex}
 \end{figure}

 \begin{figure}
     \centering
     \includegraphics[width=5.5 in]{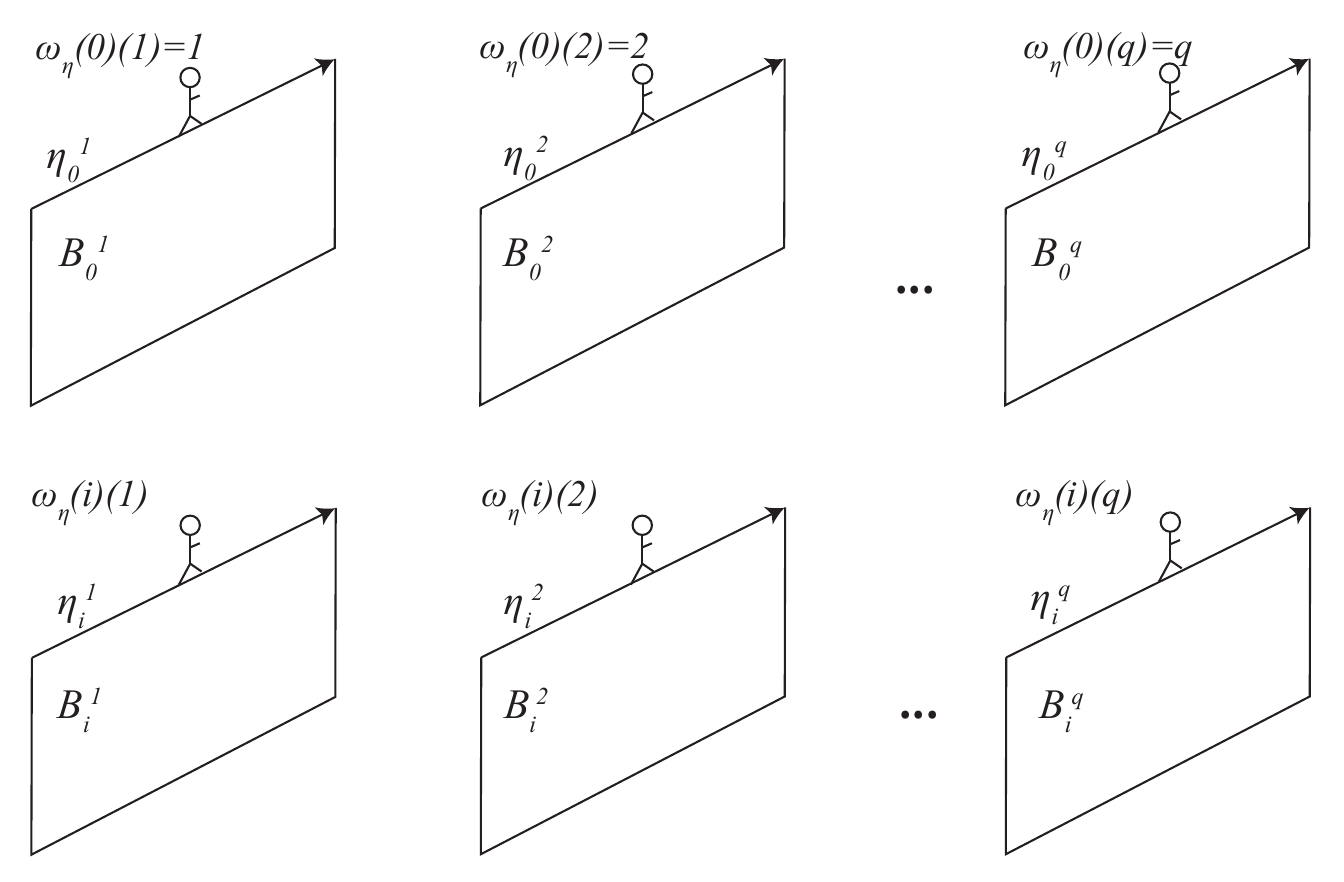}
     \caption{If the Ambler stands on the lift $\eta_i^j$ of $\eta_i$ above the 2-cell $B_i^j$, his head is in the 3-cell $E^{\omega_\eta(i)(j)}$.  Since there are $q$ lifts of each arc $\eta_i$, $\omega_\eta(i)$ is a permutation in the symmetric group $S_q$.  We assume that $\omega_\eta(0)$ is the identity permutation. }
     \label{fig:omega_eta}
 \end{figure}

It remains to label the lifts of the arcs $\eta_i$ and of the 2-cells $B_i$. First observe that the pre-image of $\eta$ in $\Sigma_q(K)$ has $q$ path-lifts, which we label $\eta^1,\dots,\eta^q$.  These path-lifts fit together to form closed, connected components. Recall that we refer to each such component as a pseudo-branch curve.

We next remark that $\forall i,j$ each component $\eta^i$ contains a lift of the arc $\eta_j$, denoted $\eta^i_j$. This 1-cell bounds a 2-cell $B_j^i$, one of the lifts of $B_j$. We would like to have a systematic way of enumerating the components $\eta^1,\dots,\eta^q$. We choose the convention that the arc $\eta^r_0$ is incident to the 3-cell $E^r$. That is, if the Ambler stands on $B^r_0$ with $\eta^r_0$ on his left, his head is in $E^r$. The component $\eta^r$ is by definition the one containing the arc $\eta^r_0$.

The functions $\omega_\eta(i)(j)$ and $\omega_\gamma(i)(j)$ are defined similarly to $\omega_K(i)(j)$ above. In other words, assume the Ambler is standing on the arc $\eta_i^j$, above the 2-cell $B_i^j$. Then, his head is in $\omega_\eta(i)(j)$.  See Figure ~\ref{fig:omega_eta}.

We make one additional remark about the positions of 2-cell $B_i$. Recall we chose the convention that if the Ambler stands on $\eta_0^k$ above $B_0^k$, his head is in $E^k$. Additionally, we assume $B_i^k$ shares a (vertical) boundary 1-cell with $B_{i+1}^k$. In other words, if the Ambler is walking to the right of $\eta$ along a push-off of $\eta$ in the direction determined by the orientation of $\eta$, when he reaches a vertical wall, he crosses from $B_{i}^k$ into $B_{i+1}^k$. 

\begin{figure}
    \centering
    \includegraphics[width=6in]{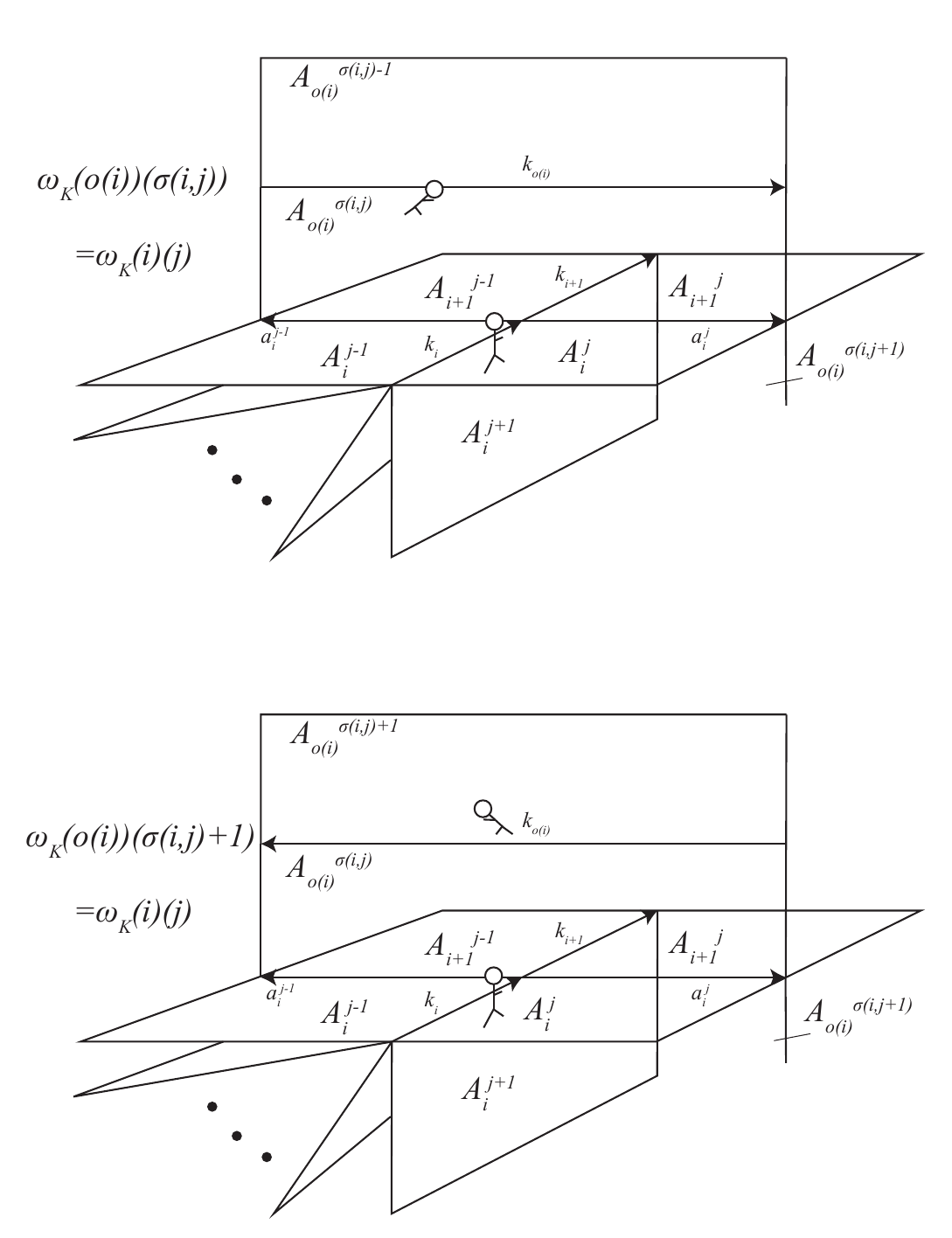}
    \caption{Lift of the cell structure at a positive (top) and negative (bottom) crossing of the branch curve $K$ over the end of arc $k_i$, with Ambler. }
    \label{fig:knot_over_knot}
\end{figure}

\begin{figure}
    \centering
    \includegraphics[width=6in]{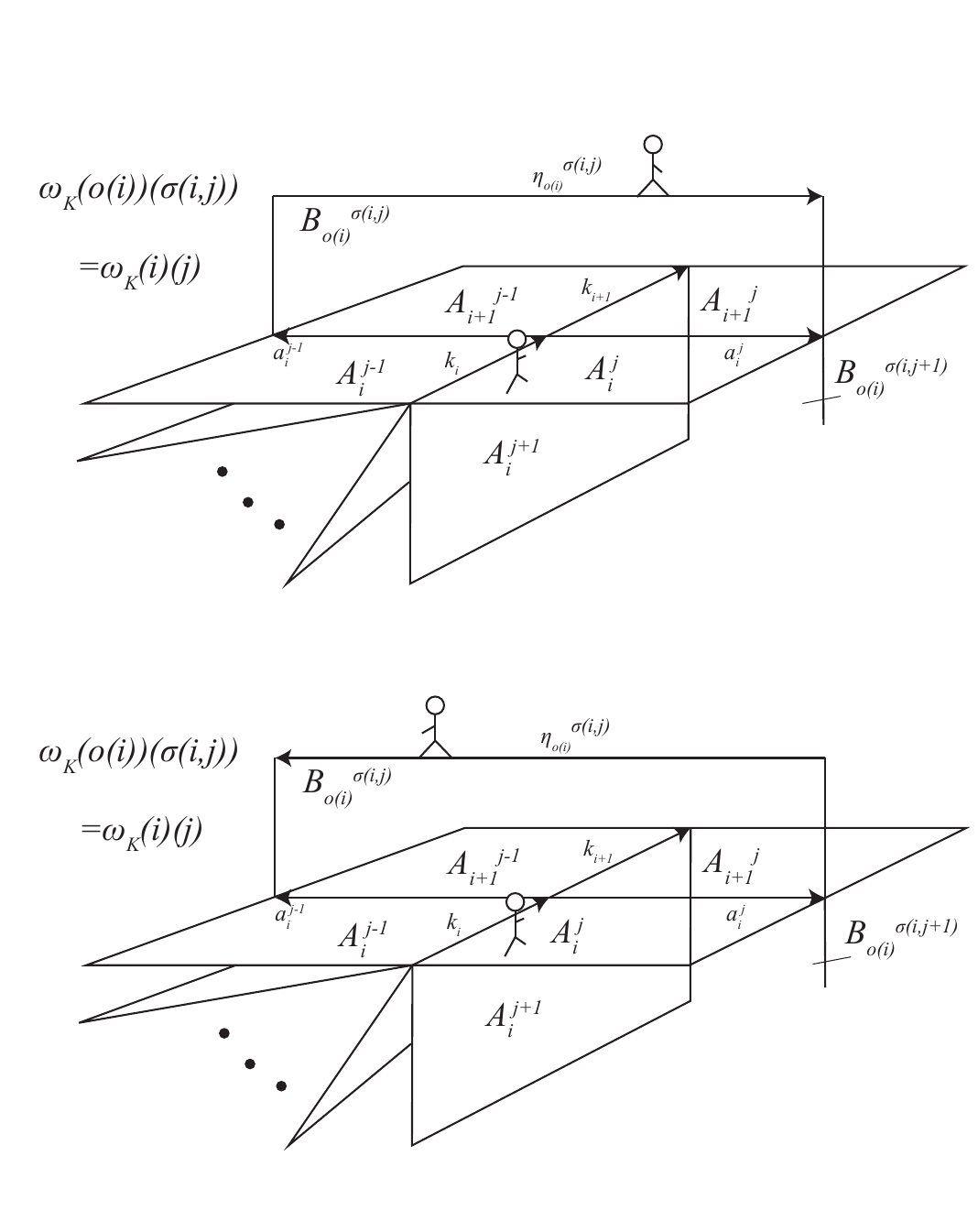}
    \caption{Lift of the cell structure at a positive (top) and negative (bottom) crossing of the pseudo-branch curve $\eta$ over the end of arc $k_i$, with Ambler. }
    \label{fig:pb_over_knot}
\end{figure}
Now consider a crossing $c_i$ in the given diagram of $K\cup\eta$, where the understrand $k_i$ and the overstrand $k_{o(i)}$ are both arcs of the branching set $K$. The crossing has $q$ lifts to $\Sigma_q(K)$ and the vertical 2-cells $A_{o(i)}^j$ are distributed across them. We introduce a function, $\sigma(i,j)$, to denote which lift of $A_{o(i)}$ lies over a given lift of $A_i$. Here, the subscript $o(i)$ denotes the overstrand at crossing $i$ of the diagram. The superscript $\sigma(i,j)$ specifies which lift of the downstairs 2-cell $A_{o(i)}$ the Ambler will cross when passing from $A_i^j$ to $A_{i+1}^j$ in $\Sigma_q(K)$. Put yet another way, the 3-cells $E^{\omega_K(i)(j)}$ and $E^{\omega_K(i+1)(j)}$ share the 2-cell $A_{o(i)}^{\sigma(i,j)}$ as part of their boundary. 

The case where the over-strand at the head of $k_i$ is an arc of $\eta$ is treated similarly. If the overstrand is  $\eta_{o(i)}$, the relevant vertical wall in $S^3$ is  $B_{o(i)}$ and its lifts $B_{o(i)}^{\sigma(i,j)}$ are labeled in the analogous fashion. That is, if the Ambler is standing on $A_i^j$ with the knot on his left and walking in the direction of the orientation of the knot, the vertical wall he will cross when he reaches a lift of crossing $i$ is $B_{o(i)}^{\sigma(i,j)}$. 

We define a function $\sigma_{\gamma}(i,j)$ in the analogous fashion, to keep track of the superscripts of the vertical 2-cells the Ambler will cross when traveling along a lift of $\gamma$.   Note that we do not define a function $\sigma_\eta(i, j)$ since, in computations, we never travel along $\eta$ unless $\gamma=\eta,$ as is the case in many applications. As a general principle, we do not name cells which do not come up in the computation. If interested in a 2-chain bounded by $\gamma,$ and 2-cells intersected when traveling along a lift of $\eta$, we simply swap the roles of $\eta$ and $\gamma$ in the labeled diagram. 

The values of the functions $\sigma(i,j)$ and  $\sigma_\gamma(i,j)$ are determined by the following formulas.

\begin{defn} The function $\sigma(i,j)$ is defined by:
\[\sigma(i,j)=\begin{cases}
    \omega_K(o(i))^{-1}[\omega_K(i)(j)]&\text{if  }\epsilon(i)=1 \text{ and }k_i \text{ terminates at the arc }k_{o(i)}\\
    \omega_K(o(i))^{-1}[\omega_K(i)(j)]-1&\text{if  }\epsilon(i)=-1\text{ and }k_i \text{ terminates at the arc }k_{o(i)}\\
    \omega_\eta(o(i))^{-1}[\omega_K(i)(j)]&\text{if }k_i \text{ terminates at the arc } \eta_{o(i)}\\
\end{cases}\]
The output of $\sigma(i,j)$ is taken mod $q$ with values between 1 and $q$.
\end{defn}

\begin{defn}The function $\sigma_\gamma(i,j)$ is defined by:
    \[
    \sigma_\gamma(i,j)=\begin{cases}
        \omega_K(o_\gamma(i))^{-1}[\omega_\gamma(i)(j)]&\text{if }\epsilon_\gamma(i)=1\text{ and }\gamma_i\text{ terminates at the arc }k_{o_\gamma(i)}\\
        \omega_K(o_\gamma(i))^{-1}[\omega_\gamma(i)(j)]-1&\text{if }\epsilon_\gamma(i)=-1\text{ and }\gamma_i\text{ terminates at the arc }k_{o_\gamma(i)}\\
        \omega_\eta(o(i))^{-1}[\omega_\gamma(i)(j)]&\text{if }\gamma_i \text{ terminates at the arc } \eta_{o_\gamma(i)}\\
    \end{cases}
    \]
    The output of $\sigma_\gamma(i,j)$ is taken mod $q$ with values between 1 and $q$. 
\end{defn}

The notation is summarized in Table~\ref{notation.tab}.

\begin{table}[htbp] 
\begin{center}
	\begin{tabular}{|l|l|l|}
		\hline

 $K$  & branching set \\
 \hline
 $q$   & degree of branched cover (in applications, $q=p^s$ for some prime $p$)  \\
\hline
 $\Sigma_q(K)$  & $q$-fold cover of $S^3$ branched along $K$ \\
 \hline
  $\eta, \gamma$  & simple closed curves in $S^3\backslash K$ (co-called pseudo-branch curves)\\
 \hline
  $n$  & number of strands of $K$ in diagram of $K\cup \eta$ \\
 \hline
 $k_i$  & strand of $K$ and horizontal 1-cell in either $S^3$ or $\Sigma_q(K)$ \\ 
 \hline
 $\eta_i$  & strand of $\eta$ and horizontal 1-cell in either $S^3$ or $\Sigma_q(K)$ \\ 
 \hline
 $a_i$  & vertical 1-cell in $S^3$, at crossing $i$ \\
 \hline
  $o(i)$  & subscript of the overstrand at crossing $i$ \\
 \hline
  $\epsilon(i)\in \{\pm 1\}$  & sign of crossing at the head of the arc $k_i$ \\
 \hline
  $\epsilon_\gamma(i)\in \{\pm 1\}$  & sign of the crossing at the head of the arc $\gamma_i$\\
 \hline
 $A_i$  & 2-cell in $S^3$ with $\partial(A_i)=k_i + a_{i+1}-a_i$ \\
 \hline
 $A_i^1,\dots A_i^q$ & lifts of $A_i$ to $\Sigma_q(K)$, with indices increasing clockwise around $K_i$\\
   \hline
    $\eta^1, \eta^2, \dots \eta^q$& closed components of the lift of $\eta$ to $\Sigma_q(K)$ \\
    \hline
    $\eta_i^j$& lift of the arc $\eta_i$ contained in the component $\eta^j$\\
 \hline
 $m$  & number of strands of $\eta$ in diagram of $K\cup \eta$ \\
 \hline
$B_i$  & 2-cell whose horizontal boundary is $\eta_i$ \\
 \hline
  $B_i^1,\dots B_i^q$ & lifts of $B_i$ to $\Sigma_q(K)$, with $B_i^k$ having $\eta_i^k$ on its boundary\\
   \hline
 $\omega_K(i)$& permutation associated to arc $K_i$ \\
  \hline
  $\omega_K(i)(x)=y$& when the Ambler stands on $A_i^x$ with $k_i$ on his left, his head is in $E^y$ \\
 \hline
 $\omega_\eta(i)(j)$ & 3-cell containing the Ambler's head when he stands on the lift $B_i^j$ of $B_i$  \\
 \hline
 $\omega_\gamma(i)$ & permutation of the 3-cells associated to arc $\gamma_i$, analogous to $\omega_\eta(i)$ \\
 \hline
$\sigma(i,j)$ & superscript on the vertical wall ($A_{o(i)}^{\sigma(i,j)}$ or $B_{o(i)}^{\sigma(i,j)}$) separating  $A_i^j$ and $A_{i+1}^j$\\
 \hline
$\sigma_\gamma(i,j)$ & superscript on the vertical wall ($A_{o(i)}^{\sigma_\gamma(i,j)}$ or $B_{o(i)}^{\sigma_\gamma(i,j)}$) 
separating $\gamma_i^j$ and $\gamma_{i+1}^j$\\

   \hline
    $\Sigma^k$& rational two-chain with boundary $\eta^k$ \\
    \hline
  $x_i^j$& coefficient of $A_i^j$ in a given 2-chain $\Sigma^k$ \\

  \hline

	\end{tabular}
	\vskip .1in
	\caption{Notation}\label{notation.tab}
\end{center} 
\end{table}

\section{Main Theorem, Specialized Version}\label{sec:thms}

To simplify the exposition, we first tailor the theorem to our primary application. That is, we begin by assuming that each of $\eta$ and $\gamma$ have $q$ closed lifts, $\eta^1,\dots, \eta^q$ and $\gamma^1,\dots,\gamma^q$.   This happens precisely when $q$ divides the linking number of both $\eta$ and $\gamma$ with $K$. We give the general linking number formula, with no assumption on the linking numbers in $S^3$, in Section~\ref{general.sec}.

Given a link $K\cup\eta\cup\gamma$ as above, for each lift $\eta^k$ of $\eta$  and $\gamma^j$ of $\gamma$ to $\Sigma_q(K)$, we check if there exists a rational two-chain with boundary $\eta^k$ and $\gamma^j$, respectively. If such chains exist, we find an explicit formula for the chain bounding $\eta^k$ and use it to compute the linking number of $\eta^k$ with $\gamma^j$.

A priori, a (rational) 2-chain bounding $\eta^k$, the $k^{\text{th}}$ lift of $\eta$, if it exists, takes the form
$$\Sigma^k= \sum_{i=0}^{m-1} z_i^kB_i^k +\sum_{i=0}^{n-1} x_i^j A_i^j.$$
The 1-cell $\eta_i^k$ appears as a boundary of only one 2-cell, $B_i^k$. This implies that the coefficient of $\eta_i^k$ in $\partial \Sigma^k$ is $z_i^k$. Since $\partial \Sigma^k=\eta^k=\Sigma_{t=0}^{m-1} \eta_t^k$, we have $z_i^k=1$.  Thus, we can rewrite the above chain as:
\begin{equation}\label{eq:2-chain-generic}
\Sigma^k= \sum_{i=0}^{m-1} B_i^k +\sum_{i=0}^{n-1} x_i^j A_i^j.    
\end{equation}

Additionally, for any arc $k_i$ of $K$, the coefficient of $k_i$ in $\partial \Sigma^k$ is $\sum_{j=1}^q x_i^j$. Since $\partial \Sigma^k=\eta^k$, we conclude that $\sum_{j=1}^q x_i^j=0$ for each $i$.

The remaining constraints, a system of linear equations in the $x_i^j$, are summarized in Theorem~\ref{thm:mainA}.

\begin{thm} \label{thm:mainA}
Let  $K\cup\eta$ be a link in $S^3$. Let $q$ be an integer dividing $lk(K, \eta)$. We use the notation of Section~\ref{sec:cellular}. The lift $\eta^k$ of $\eta\in S^3(K)$ to $\Sigma_q(K)$, if rationally nullhomologous, bounds the 2-chain $\Sigma^k$ in Equation~\ref{eq:2-chain-generic}, where the $x_i^j$ satisfy $\sum_{j=1}^q x_i^j=0$ for each $i$, and are simultaneously a solution to the following system:
\[\begin{cases}
x_i^j-x_{i+1}^j-\epsilon(i)(x_{o(i)}^{\sigma(i,j)}-x_{o(i)}^{\sigma(i,j+1)})=0&\text{if }k_i\text{ terminates at the arc } k_{o(i)}\\
x_i^j-x_{i+1}^j=\epsilon(i) &\text{if } k_i
 \text{ terminates at the arc } \eta_{o(i)} \text{ and }\sigma(i,j)=k\\
 x_i^j-x_{i+1}^j=-\epsilon(i) &\text{if } k_i
 \text{ terminates at the arc } \eta_{o(i)} \text{ and }\sigma(i,j+1)=k\\
x_i^j-x_{i+1}^j=0&\text{if } k_i \text{ terminates at the arc } \eta_{o(i)} \text{ and } \sigma(i,j),\sigma(i,j+1)\neq k,\\
 \end{cases}\]
\noindent with superscripts taken mod $q$ between $1$ and $q$, and subscripts taken mod $n$ between $0$ and $n-1$.  If the system of equations above has no solution over $\mathbb{Q}$,  $\eta^k$ is non-zero in $H_1(\Sigma_q; \mathbb{Q}).$ 
\end{thm}

\begin{proof}  
Suppose $\partial \Sigma^k=\eta^k$, with
$$\Sigma^k= \sum_{i=0}^{m-1} B_i^k +\sum_{i=0}^{n-1} x_i^j A_i^j.$$  

Let $i$ be the crossing where the arc $k_i$ terminates. We compute the contribution to $\partial \Sigma^k$ by cells in the lift of a neighborhood of the crossing $i$ to $M$, which we denote by $\partial|_i \Sigma^k$.  (We will consider 4 different types of crossings, as in the theorem statement; the first case corresponds to the first crossing type, and the second case covers the other three crossing types.)

{\it Case 1:} At crossing $i$, $k_i$ terminates at the arc $k_{o(i)}$. The cell structure in the lift of a neighborhood of a crossing $i$ of this type is shown in Figure ~\ref{fig:knot_over_knot}.  

 In this case, $\partial|_i\Sigma^k$ takes the form

$$\partial|_i\Sigma^k=\left(\sum_{j=1}^q x_i^j\right) k_i+\left(\sum_{j=1}^q x_{i+1}^j\right)k_{i+1}+\left(\sum_{j=1}^q x_{o(i)}^j\right)k_{o(i)}+\sum_{i=1}^q  y_i^j a_i^j,$$

where $y_i^j$ is a linear combination of the $x_r^s$, determined by the configuration of cells at the given crossing.  As previously established (see paragraph preceding the theorem statement), $\sum_{j=1}^q x_i^j=0$,  $\sum_{j=1}^q x_{i+1}^j=0$, and  $\sum_{j=1}^q x_{o(i)}^j=0$.
In addition, each linear combination $y_i^j$ of the $x_r^s$'s is equal to 0, since $y_i^j$ is the coefficient of $a_i^j$ in $\partial \Sigma^k=\eta^k$.

We now express $y_i^j$ in terms of the $x_r^s$.  Refer again to Figure ~\ref{fig:knot_over_knot}, where both cases, $\epsilon(i)=1$ and $\epsilon(i)=-1$, are depicted.

First suppose $\epsilon(i)=1$.  In this case, $a_i^j$ is incident to $A_i^j$, $A_{i+1}^j$, $A_{o(i)}^{\sigma(i,j)}$, and $A_{o(i)}^{\sigma(i,j+1)}$, and appears in the boundaries of these 2-cells with coefficients $x_i^j$, $-x_{i+1}^j$, $-x_{o(i)}^{\sigma(i,j)}$, and $x_{o(i)}^{\sigma(i,j+1)}$.  Therefore, $y_i^j=x_i^j-x_{i+1}^j-x_{o(i)}^{\sigma(i,j)}+x_{o(i)}^{\sigma(i,j+1)}$.  

The case $\epsilon(i)=-1$ is analogous, and we conclude that $y_i^j=x_i^j-x_{i+1}^j+x_{o(i)}^{\sigma(i,j)}-x_{o(i)}^{\sigma(i,j+1)}$.  

Setting $y_i^j=0$ in both cases above results in the equation $$x_i^j-x_{i+1}^j-\epsilon(i)(x_{o(i)}^{\sigma(i,j)}-x_{o(i)}^{\sigma(i,j+1)})=0,$$

which is the first equation in the theorem statement, as desired.

{\it Case 2:} At crossing $i$, $k_i$ terminates at the arc $\eta_{o(i)}$.  The cell structure in the lift of a neighborhood of a crossing $i$ of this type is shown in Figure ~\ref{fig:pb_over_knot}.

In this case,

$$\partial|_i \Sigma^k=\left(\sum_{j=1}^q x_i^j\right) k_i+\left(\sum_{j=1}^q x_{i+1}^j\right)k_{i+1}+\eta_{o(i)}^k+\sum_{i=1}^q  y_i^j a_i^j,$$

where again $y_i^j$ is a linear combination of the $x_r^s$.  As in Case 1, the first two sums are zero, and each linear combination $y_i^j$ of the $x_r^s$'s is equal to 0.

We again express $y_i^j$ in terms of the $x_r^s$ and set the result equal to 0 to obtain the desired linear equations.

The case $\epsilon(i)=1$ is shown in the top of Figure ~\ref{fig:pb_over_knot}. The only 2-cell of type $B$ appearing at crossing $i$ is $B_{o(i)}^k$.  We observe that $a_i^j$ is incident to $A_i^j$, $A_{i+1}^j$, $B_{o(i)}^{\sigma(i,j)}$, and $B_{o(i)}^\sigma(i,j+1)$, and appears in the boundaries of these 2-cells with coefficients 1, -1, -1, and 1, respectively.  If $\sigma(i,j)=k$, then $y_i^j=x_i^j-x_{i+1}^{j}-1$.  If $\sigma(i,j+1)=k$, then $y_i^j=x_i^j-x_{i+1}^{j}+1$.  Otherwise, $y_i^j=x_i^j-x_{i+1}^{j}.$

The case $\epsilon(i)=-1$ is similar, and shown in the bottom of Figure ~\ref{fig:pb_over_knot}.  In this case, $a_i^j$ appears in the boundaries of $A_i^j$, $A_{i+1}^j$, $B_{o(i)}^{\sigma(i,j)}$, and $B_{o(i)}^\sigma(i,j+1)$ with coefficients 1, -1, 1, and -1 respectively.   If $\sigma(i,j)=k$, then $y_i^j=x_i^j-x_{i+1}^{j}+1$.  If $\sigma(i,j+1)=k$, then $y_i^j=x_i^j-x_{i+1}^{j}-1$.  Otherwise, $y_i^j=x_i^j-x_{i+1}^{j}.$

Putting these two cases $\epsilon(i)=\pm 1$ together, and setting $y_i^j=0$, we get the final three equations

\[\begin{cases}

x_i^j-x_{i+1}^j=\epsilon(i) &\text{if } k_i
 \text{ terminates at the arc } \eta_{o(i)} \text{ and }\sigma(i,j)=k\\
 x_i^j-x_{i+1}^j=-\epsilon(i) &\text{if } k_i
 \text{ terminates at the arc } \eta_{o(i)} \text{ and }\sigma(i,j+1)=k\\
x_i^j-x_{i+1}^j=0&\text{if } k_i \text{ terminates at the arc } \eta_{o(i)} \text{ and } \sigma(i,j),\sigma(i,j+1)\neq k,\\
 \end{cases}\]

 as desired.
\end{proof}

\begin{thm} \label{thm:mainB} Let $K\cup \eta\cup\gamma$ be a link in $S^3$. We adopt the notation and hypotheses of Theorem~\ref{thm:mainA}, and in particular assume that $q$ divides the linking numbers of $\eta$ and $\gamma$ with $K$.  In addition, if $\gamma=\eta$, we assume in what follows that $j\neq k$.
    If $\gamma^j$ and $\eta^k$ are rationally null-homologous in $\Sigma_q(K)$, the linking number of the lift $\gamma^j$ with the lift $\eta^k$, with $\eta^k=\partial \Sigma^k$ defined as in Theorem~\ref{thm:mainA} is
        $\text{lk}(\gamma^j,\eta^k)=\sum_{i=0}^{s-1}a_i,$   where    
\begin{equation}\label{linking-formula}
     a_i=
        \begin{cases}
            \epsilon_\gamma(i)\cdot x_{o_\gamma(i)}^{\sigma_\gamma(i,j)}&\text{if }\gamma_i\text{ terminates at the arc }k_{o_\gamma(i)}\\
            \epsilon_\gamma(i)&\text{if }\gamma_i\text{ terminates at the arc }\eta_{o_\gamma(i)}\text{ and }\sigma_\gamma(i,j)=k\\
            0&\text{otherwise}
        \end{cases}
    \end{equation}
    
    \end{thm}
   
\begin{figure}
    \centering
    \includegraphics[width=6in]{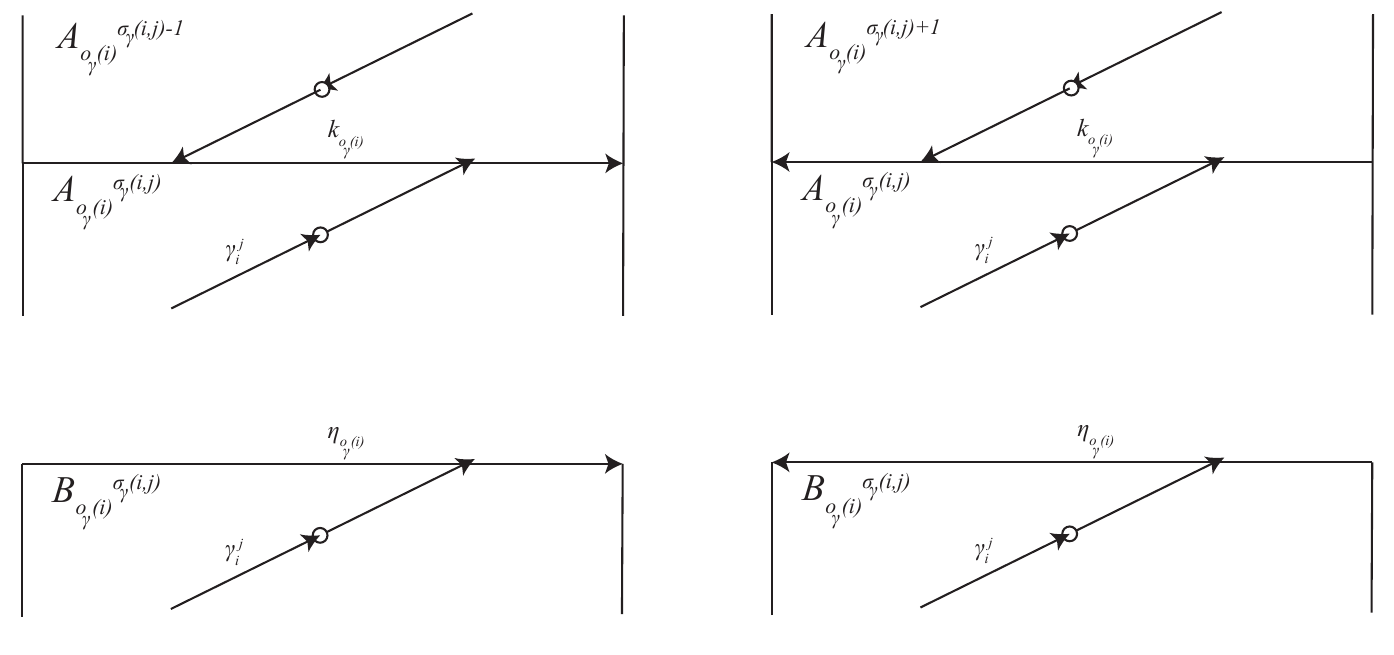}
    \caption{The lift second pseudo-branch curve $\gamma$ and the cell structure when $\gamma_i$ terminates at an arc of the knot $K$ (top, positive and negative crossings) and at an arc of the first pseudo-branch curve $\eta$ (bottom, positive and negative crossings).}
    \label{fig:pb2_under_knot_or_pb1}
\end{figure}

\begin{proof}
  Recall that, if $\eta^k$ bounds the rational chain $\Sigma^k$, the linking number $\text{lk}(\gamma^j,\eta^k)$ equals the signed intersection number between $\gamma^j$ and $\Sigma^k$. This number is symmetric and independent of the choice of two-chain~\cite{birman1980seifert}.  To compute the intersection number, we first need to ensure that $\gamma^j$ and $\Sigma^k$ are transverse. Recall that we started with a generic projection of the link diagram and we obtained the 2-skeleton by taking the cone on $K\cup \eta$; we then lifted this cell structure to $\Sigma_q(K)$. Therefore, by construction, $\gamma^j$, a lift of $\gamma$, is transverse to all 2-cells in $\Sigma_q(K)$ except (in the case where $\gamma=\eta$) those it bounds. By assumption, $\gamma^j\neq \eta^k$ and therefore, by Equation~\ref{eq:2-chain-generic}, $\gamma^j$ is not on the boundary of any of the 2-cells in $\Sigma^k$. In sum, $\gamma^j$ is transverse to all cells in the 2-chain $\Sigma^k$.
   
  Recall that $\gamma^j$ can be regarded as a boundary union of embedded arcs, $\gamma^j=\sum_{i=0}^{s-1}\gamma_i^j$. These endpoints are lifts of crossings in the diagram where $\gamma$ passes under $K$ or $\eta$. 
  Moreover,  intersections between $\gamma^j$ and $\Sigma^k$ can only occur at the endpoints of these arcs. 
  
  For each pair $i, j$, we consider endpoint where the arc $\gamma_i^j$ terminates. Which type ($A$ or $B$) of vertical 2-cell this endpoint belongs to is determined by the link component ($K$ or $\eta$, respectively) containing the over-strand at the crossing where the $i$-th strand of $\gamma$ ends. The subscript of the 2-cell is determined by the function $o_\gamma(i)$, and the superscript by the function $\sigma_\gamma(i,j)$. The coefficient of this 2-cell in $\Sigma^k$ equals  $x_{o_\gamma(i)}^{\sigma_\gamma(i,j)}$ if it is an $A$ cell; $1$ if it is a $B$ cell with horizontal boundary belonging to $\eta^k$; 0 otherwise. (The reader may refer again to Equation~\ref{eq:2-chain-generic}.)  Finally, the sign of the intersection between the 2-cell and $\gamma^j$ is determined by the function $\epsilon_\gamma(i)$. The term $a_i$ in the theorem statement is computed from the above data: it equals 0 if the 2-cell where $\gamma_i^j$ terminates is not part of $\Sigma^k$; otherwise it equals the coefficient of this 2-cell times the sign of the crossing. 
\end{proof}

\section{Main Theorem, General Case}\label{general.sec}

We now drop the hypothesis that $q$ divides the linking numbers $\text{lk}(K,\eta)$ and $\text{lk}(K,\gamma)$.

As before, $\eta$ (respectively $\gamma$) has $q$ path-lifts, $\eta^1,\dots, \eta^q$, labeled so that $\eta_0^k$ is in the 3-cell $E^k$; but the $\eta^k$ may not be closed loops.  However, the $\eta^i$ fit together to form closed components, what we call pseudo-branch curves, and we can compute linking numbers between these closed components and the corresponding closed components for $\gamma$.

First, observe that the number, $I_\eta$, of connected components in the lift of $\eta$ equals the index, in $\mathbb{Z}/q\mathbb{Z}$, of the subgroup generated by $\text{lk}(K,\eta)$. 
Traveling along a component of the lift of $\eta$, one will encounter path-lifts $\eta^r, \eta^s, \dots$ whose indices jump by $\text{lk}(K,\eta) \mod q =:\bar{l}_\eta$, until they begin to repeat. 
The pseudo-branch curves $\boldsymbol{\eta}^k$ projecting to $\eta$ are therefore
$$\boldsymbol{\eta}^k=\eta^k \cup \eta^{k+\bar{l}_\eta} \cup \dots \cup \eta^{k+(q/I-1)\cdot \bar{l}_\eta} ,$$
where $k\in\{1,\dots, I_\eta\}$.  The superscripts correspond to elements of the coset $G_\eta^k=k+\langle \bar{l}_\eta\rangle$ in $\mathbb{Z}/q\mathbb{Z}$, with values taken in $\{1,\dots,q\}$. 
We define the closed component $\boldsymbol{\gamma}^j$  
as well as $G_\gamma^j$ 
analogously, with $j\in\{1,\dots, I_\gamma\}$.

Reasoning as we did in order to arrive at Equation~\ref{eq:2-chain-generic}, we see that a generic 2-chain bounding $\boldsymbol{\eta}^k$, if it exists, takes the form 
\begin{equation}\label{eq:2-chain-generic-general}
\boldsymbol{\Sigma}^k= \sum_{k'\in G_\eta^k}\left(\sum_{i=0}^{m-1} B_i^{k'}\right) +\sum_{i=0}^{n-1} x_i^j A_i^j.    
\end{equation}

\begin{thm} \label{thm:mainAgeneral}
Let  $K\cup\eta$ be a link in $S^3$. We use the notation of Section~\ref{sec:cellular}. The pseudo-branch component $\boldsymbol{\eta}^k$ of $\eta\in S^3(K)$ in $\Sigma_q(K)$, if rationally nullhomologous, bounds the 2-chain $\boldsymbol{\Sigma}^k$ in Equation~\ref{eq:2-chain-generic-general}, where the $x_i^j$ satisfy $\sum_{j=1}^q x_i^j=0$ for each $i$, and are simultaneously a solution to the following system:
\[\begin{cases}
x_i^j-x_{i+1}^j-\epsilon(i)(x_{o(i)}^{\sigma(i,j)}-x_{o(i)}^{\sigma(i,j+1)})=0&\text{if }k_i\text{ terminates at the arc } k_{o(i)}\\
x_i^j-x_{i+1}^j=\epsilon(i) &\text{if } k_i
 \text{ terminates at the arc } \eta_{o(i)}, \\ 
 &\sigma(i,j)\in G_\eta^k,\text{ and }\sigma(i,j+1)\notin G_\eta^k\\
 x_i^j-x_{i+1}^j=-\epsilon(i) &\text{if } k_i
 \text{ terminates at the arc } \eta_{o(i)},\\
 &\sigma(i,j+1)\in G_\eta^k,\text{ and }\sigma(i,j)\notin G_\eta^k \\
x_i^j-x_{i+1}^j=0&\text{if } k_i \text{ terminates at the arc } \eta_{o(i)}\text{ and } \\
&\{\sigma(i,j),\sigma(i,j+1)\}\subset G_\eta^k 
\text{ or } \{\sigma(i,j),\sigma(i,j+1)\}\cap G_\eta^k=\emptyset\\
 \end{cases}\]
\noindent with superscripts taken mod $q$ between $1$ and $q$, and subscripts taken mod $n$ between $0$ and $n-1$.  If the system of equations above has no solution over $\mathbb{Q}$,  $\boldsymbol{\eta}^k$ is non-zero in $H_1(\Sigma_q; \mathbb{Q}).$ 
\end{thm}
The proof of Theorem ~\ref{thm:mainAgeneral} is analogous to the proof of Theorem ~\ref{thm:mainA} but this time taking into account the way different path lifts combine to form connected components.

Now to compute the linking number of $\boldsymbol{\gamma}^j$ with $\boldsymbol{\eta}^k$, we first use Theorem ~\ref{thm:mainAgeneral} to find a rational 2-chain $\boldsymbol{\Sigma}^k$ bounding $\boldsymbol{\eta}^k$, and to verify that $\boldsymbol{\gamma}^j$ is rationally null-homologous.  We then determine the linking numbers according to the following formula:

\begin{thm} \label{thm:mainBgeneral} Let $K\cup \eta\cup\gamma$ be a link in $S^3$. We adopt the notation and hypotheses of Theorem~\ref{thm:mainAgeneral}.  In addition, if $\gamma=\eta$, we assume in what follows that $j\neq k$.
    If $\boldsymbol{\gamma}^j$ and $\boldsymbol{\eta}^k$ are rationally null-homologous in $\Sigma_q(K)$, the linking number of the lift $\boldsymbol{\gamma}^j$ with the lift $\boldsymbol{\eta}^k$, with $\boldsymbol{\eta}^k=\partial \boldsymbol{\Sigma}^k$ defined as in Theorem~\ref{thm:mainAgeneral} is
        $\text{lk}(\boldsymbol{\gamma}^j,\boldsymbol{\eta}^k)=\sum_{j'\in G_\gamma^j}\left(\sum_{i=0}^{s-1}a_{i,j'}\right),$   where    
\begin{equation}\label{linking-formula-general}
     a_{i,j'}=
        \begin{cases}
            \epsilon_\gamma(i)\cdot x_{o_\gamma(i)}^{\sigma_\gamma(i,j')}&\text{if }\gamma_i\text{ terminates at the arc }k_{o_\gamma(i)}\\
            \epsilon_\gamma(i)&\text{if }\gamma_i\text{ terminates at the arc }\eta_{o_\gamma(i)}\text{ and }\sigma_\gamma(i,j')\in G_\eta^k\\
            0&\text{otherwise}
        \end{cases}
    \end{equation}
    
    \end{thm}

    The proof of Theorem ~\ref{thm:mainBgeneral} is analogous to the proof of Theorem ~\ref{thm:mainB}. The contribution of each intersection point of a path-lift of $\gamma$ and a 2-cell in $\boldsymbol{\Sigma}^k$ is computed as before; these contributions are then distributed across connected components $\boldsymbol{\gamma}^j$ as dictated by the indexing sets $G_\gamma^j$.

\section{Applications}\label{sec:apps}
One application of computing linking numbers of pseudobranch curves in cyclic branched covers is to obstruct satellite operations from inducing homomorphisms on concordance, using the following results from~\cite{lidman2022linking}. Throughout this section, $\eta$ denotes a meridian of the solid torus in which the satellite pattern is embedded. 

\begin{thm} \label{thm:obstruction-statement} \cite{lidman2022linking}
Let $P$ be a pattern in the solid torus with winding number $w$, and let $q$
be a prime power dividing $w$. Suppose that the lifts $\eta_1, \dots, \eta_q$ of $\eta$ to $\Sigma_q(P(U))$  are null-homologous in $\Sigma_q(P(U))$. If $lk(\eta_i, \eta_j )\geq 0$ for all $i\neq j$, but is not identically zero, then $P$ does not induce a homomorphism of the smooth concordance group.
\end{thm}

\begin{thm}\label{thm:obstruction-2} \cite{lidman2022linking}
Let $P$ be a pattern in the solid torus with winding number $w$, and let $q$
be a prime power dividing $w$. Let $\eta_1, \dots, \eta_q$ denote the lifts of $\eta$ to $\Sigma_q(P(U))$ and let $n$ denote the order of $[\eta_1]$ in $H_1(\Sigma_q(P(U)))$. Suppose that either $n$ is odd or $w$ is a nonzero multiple of $n$.  If $lk(\eta_i, \eta_j )\geq 0$ for all $i\neq j$, but is not identically zero, then $P$ does not induce a homomorphism of the smooth concordance group.
\end{thm}
     
In this section, we apply our theorem to evaluate, in a variety of cases, the above obstructions.
To start, we recover the (easy) fact that for $n\neq \pm 1$ the $(n,1)$ cable map does not induce a homomorphism on smooth concordance. In Example~\ref{ex:modified-cable}, we also consider certain satellites obtained by crossing changes in the standard diagram of the cabling pattern. We include these examples in part because the results are verifiable by hand (contrast Example~\ref{ex:gnarly-slice}), since the branching set, $P(U)$, is an unknot. In addition, these computations lead to the observation (Corollary~\ref{cor:obstruction-destruction}) that for any knot $K\subset S^3$ there are infinitely many embeddings of $K$ in the solid torus, of arbitrarily large winding numbers, that produce patterns for which the homomorphism obstruction in Theorem~\ref{thm:obstruction-statement} vanishes.  Next, in Example~\ref{ex:steve}, we show that a family of satellite patterns determined by the Stevedore knot (embedded in different ways in a solid torus) do not induce homomorphisms on smooth concordance.  We conclude with Example~\ref{ex:gnarly-slice}, which serves to illustrate that our method can handle knots of large Seifert genus, as well as high degree covers, without sweat.

\begin{ex} \label{ex:modified-cable} We consider the $(n,1)$ cabling operation, as well as satellites obtained by crossing changes in the standard diagrams for those cables.  The point here is to give explicit and relatively easily visualizable two-chains for the lifts of $\eta$, in a family of examples where the answer can also be obtained by a geometric argument. (Compare Example 1.3 of~\cite{lidman2022linking}.) 

\begin{figure}
    \centering
    \includegraphics[width=4in]{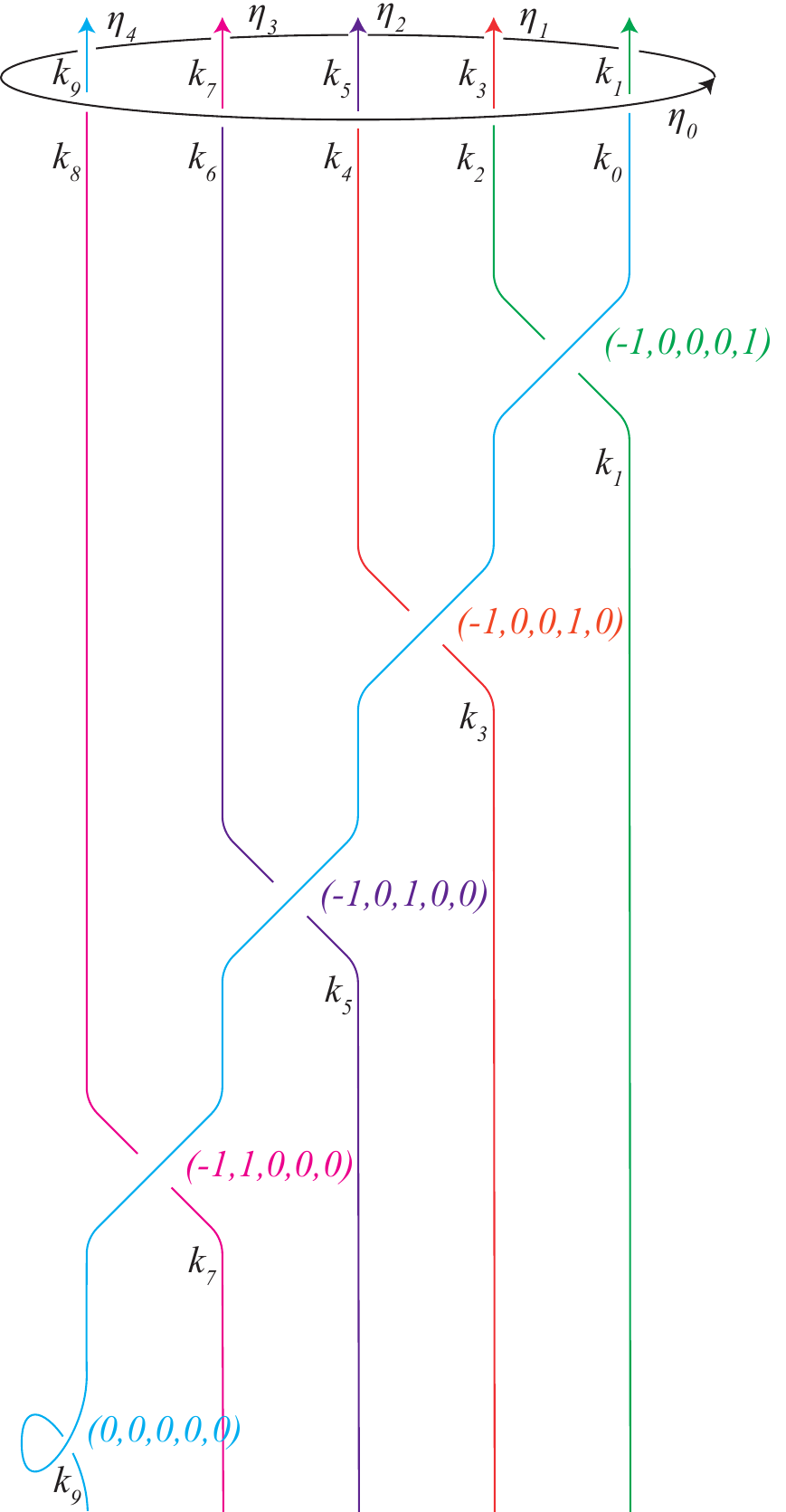}
    \caption{A numbered diagram of the $(5,1)$-cable, illustrating Equation \ref{cablechain.eq} for the case $n=5$. For each arc $k_i$ of the cable, the vector of the same color lists the coefficients $(x_i^1,x_i^2,x_i^3,x_i^4,x_i^5)$  of the 2-cells $A_i^1,\dots,A_i^5$ that appear in the 2-chain bounding $\eta_1$.} 
    \label{fig:51cable}
\end{figure}

First consider the diagram for the $(n,1)$ cable, together with a choice of $\eta$, shown in Figure ~\ref{fig:51cable} in the case $n=5$.  We will construct the 2-chain $\Sigma^1$ bounded by $\eta^1$. The 2-chains bounding the other lifts of $\eta$ can be obtained by acting on the superscripts of the coefficients of the cells in $\Sigma^1$ by a cyclic permutation.  For each arc $k_i$ in the diagram, we give a vector whose entries are the coefficients $(x_i^1,x_i^2,\dots x_i^q)$ of the 2-cells $A_i^1$, $A_i^2$,$\dots$, $A_i^q$ in the 2-chain $\Sigma^1$.  Letting $e_j$ denote the usual standard basis vector in $\mathbb{Q}^q$ with a $1$ in position $j$ and $0$'s elsewhere, we have
\begin{equation}\label{cablechain.eq}
(x_i^1,x_i^2,\dots x_i^q)=
\begin{cases}
e_{n-i/2}-e_1&\text{if } i \text{ even}\\
e_{n-(i-1)/2}-e_1&\text{if } i \text{ odd}
\end{cases}
\end{equation}

These vectors, together with their corresponding arcs $k_i$ of the cable diagram, are shown in Figure ~\ref{fig:51cable}. Lastly, $\Sigma^1$ also contains each $B_i^1$, per Equation~\ref{eq:2-chain-generic}. The 2-cells in the 2-chain described above are pictured, shaded, in Figure~\ref{fig:cable2chain}.

\begin{figure}
    \centering
    \includegraphics[width=\textwidth]{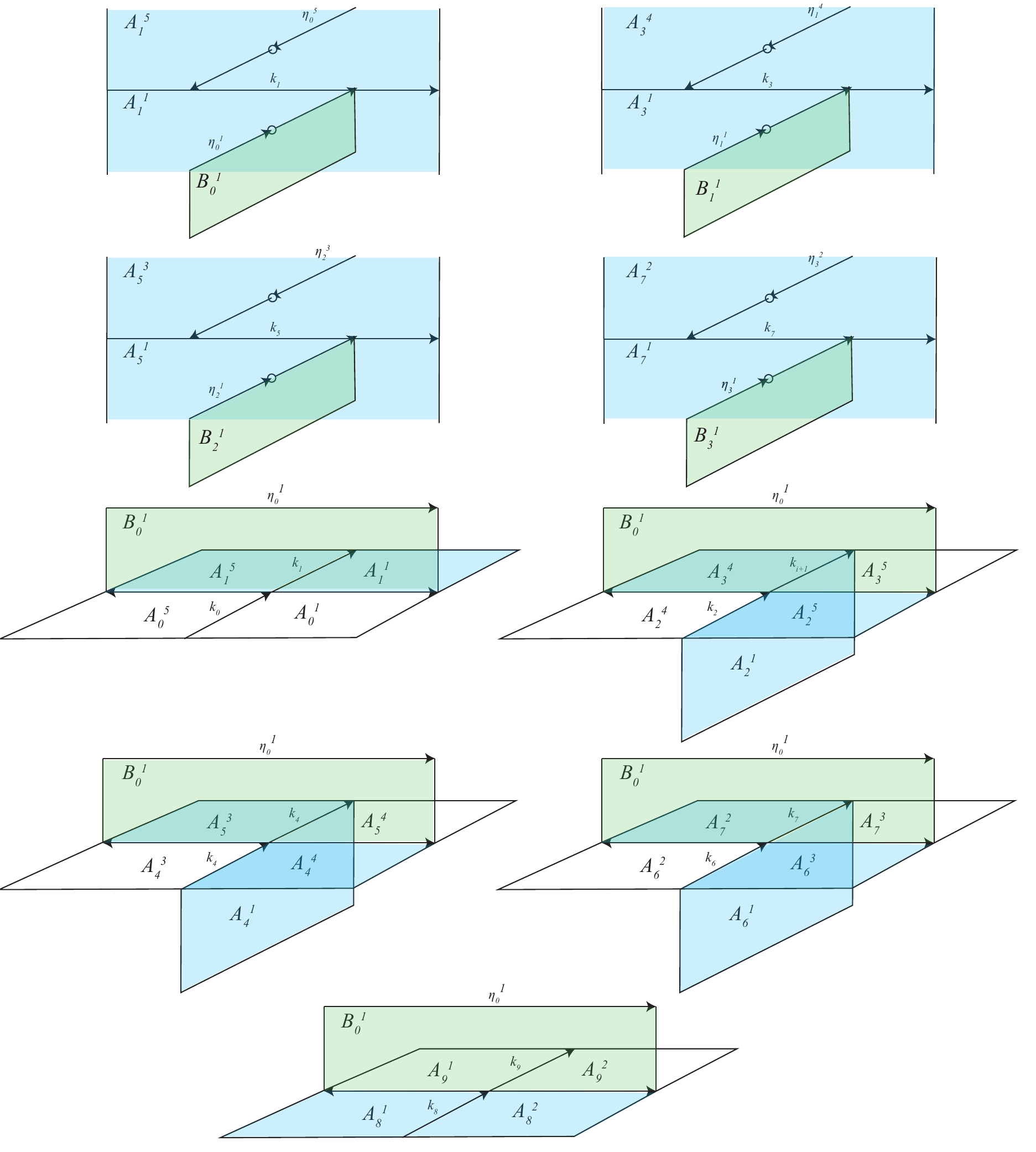}
    \caption{The blue and green 2-cells appear in a 2-chain bounding the lift $\eta^1$ of the pseudo-branch curve $\eta$ for the $(5,1)$-cable, corresponding to the numbered diagram in Figure ~\ref{fig:51cable}.}
    \label{fig:cable2chain}
\end{figure}

To compute the linking number of the lift $\eta^k$ ($k\neq 1$) with $\eta^1$, we count signed intersections of $\eta^k$ with the 2-cells in $\Sigma^1$.  Observe in Figure ~\ref{fig:cable2chain} for the case $n=5$ (an analogous picture can be drawn for the general case) that the only arc $\eta_i^k$ that terminates at a cell of $\Sigma^1$ is $\eta_{n-k}^k$.  Furthermore, the arc $\eta_{n-k}^k$ terminates at the cell $A_{2(n-k)+1}^k$, which according to Equation ~\ref{cablechain.eq} appears in $\Sigma^1$ with coefficient $+1$.  Therefore the linking number $\text{lk}(\eta^k,\eta^1)$ is equal to 1.

 See Table ~\ref{tab:cable_crossing_change} to get a sense for how the 2-chains vary for other satellites obtained from the cable by $k$ crossing changes in the case $n=3$.  See Table ~\ref{tab:crossing_change_linking_nos} to get a sense of how the linking numbers change for larger $n$ and $k$.
 
 In particular, the fact that the linking numbers of the lifts of $\eta$ for the alternating cables are 0 can be deduced by similar reasoning as applied to the standard $(n,1)$ cable above. We omit this technical argument in favor of the geometric one given below.
\end{ex}

\begin{table}[]
    \centering
    \begin{tabular}{|c|c|c|c|}
    \hline
        $k$ & Lift& Coefficients of 2-Chain for $\eta^i$ & $\text{lk}(\eta^1,\eta^i)$\\
        &$\eta^i$&$(x_0^1,x_0^2,x_0^3|x_1^1,x_1^2,x_1^3|x_2^1,x_2^2,x_2^3|x_3^1,x_3^2,x_3^3|x_4^1,x_4^2,x_4^3|x_5^1,x_5^2,x_5^3)$&\\
        \hline
        0&$\eta^1$&$(0,0,0|-1,0,1|-1,0,1|-1,1,0|0,0,0)$&$0$\\
        0&$\eta^2$&$(0,0,0|1,-1,0|1,-1,0|0,-1,1|0,-1,1|0,0,0)$&$1$\\
        0&$\eta^3$&$(0,0,0|0,1,-1|0,1,-1|1,0,-1|1,0,-1|0,0,0)$&$1$\\
        \hline
        \hline
        1&$\eta^1$&$(0,0,0|-1,2,-1|0,1,-1|1,0,-1|1,0,-1|0,0,0)$&$0$\\
        1&$\eta^2$&$(0,0,0|-1,-1,2|-1,0,1|-1,1,0|-1,1,0|0,0,0)$&$0$\\
        1&$\eta^3$&$(0,0,0|2,-1,-1|1,-1,0|0,-1,1|0,-1,1|0,0,0)$&$0$\\
        \hline
        \hline
        2&$\eta^1$&$(0,0,0|-2,1,1|0,-3,3|0,-2,2|0,-1,1|0,0,0)$&$0$\\
        2&$\eta^2$&$(0,0,0|1,-2,1|3,0,-3|2,0,-2|1,0,-1|0,0,0)$&$-1$\\
        2&$\eta^3$&$(0,0,0|1,1,-2|-3,3,0|-2,2,0|-1,1,0|0,0,0)$&$-1$\\
        \hline
          
    \end{tabular}
    \caption{Outline of linking number computation for patterns obtained from the standard diagram of the $(3,1)$ cable by changing the $k$ rightmost crossings. Refer to Table~\ref{notation.tab} for the 2-cells whose coefficients are given by the $x_i^j$.  The numbering of the arcs $k_i$ corresponding to these coefficients is analogous to that of Figure~\ref{fig:51cable}, with the arc $k_0$ in the lower left.}
    \label{tab:cable_crossing_change}
\end{table}

\begin{table}[]
    \centering
    \begin{tabular}{|c|c|c|}
    \hline
         $n$&$k$&$\text{lk}(\eta^1,\eta^2),\text{lk}(\eta^1,\eta^3),\dots,\text{lk}(\eta^1,\eta^n)$  \\
         \hline
         $3$& $0$ & $1,1$\\
         $3$& $1$ & $0,0$\\
         $3$& $2$ & $-1,-1$\\
         \hline\hline
          $5$& $0$ & $1,1,1,1$\\
         $5$& $1$ & $0,1,1,0$\\
         $5$& $2$ & $0,0,0,0$\\
         $5$& $3$ & $0,-1,-1,0$\\
         $5$& $4$ & $-1,-1,-1,-1$\\
         \hline\hline 
         $7$& $0$ &$1,1,1,1,1,1$\\
         $7$& $1$ & $0,1,1,1,1,0$\\
         $7$& $2$ & $0,0,1,1,0,0$\\
         $7$& $3$ &$0,0,0,0,0,0$\\
         $7$& $4$ &$0,0,-1,-1,0,0$\\
         $7$& $5$ &$0,-1,-1,-1,-1,0$\\
         $7$& $6$ &$-1,-1,-1,-1,-1,-1$\\
         \hline
    \end{tabular}
    \caption{Linking numbers for the pattern $C_{n,k}$ obtained from the $(n,1)$ cable by $k$ crossing changes.}
    \label{tab:crossing_change_linking_nos}
\end{table}

It is implicit in~\cite{miller2023homomorphism, lidman2022linking} that alternating cabling patterns do not induce homomorphisms on concordance, and also that the linking number obstruction vanishes for these patterns. Indeed, if $P(U)$ is isotopic to $-P(U)$ inside $S^1\times D^2$, then the linking numbers between lifts of $\eta$ to a cyclic branched cover of $P(U)$ must be zero: on the one hand, mirroring changes the sign of these linking numbers (while the choice of orientation on the branching set leaves the numbers unchanged); on the other hand, the pattern is isotopic to its inverse. In other words, the homomorphism obstruction vanishes for patterns which are amphichiral inside the solid torus. This property is not affected by introducing local knotting, as we now recall. 
 
The following is well known.
\begin{lem} \label{lem:local-knotting} 
Let $K\subset S^3$ be a knot and $\eta, \gamma\subset S^3\backslash K$ two disjoint simple closed curves. Let $S$ be a 2-sphere embedded in $S^3$ so that the Heegaard splitting determined by $S$, $S^3=B^3_1\cup_S B^3_2$, has the property that ${B}^3_2\cap K$ is an unknotted arc $\xi$  properly embedded in $B^3_2$; and $B^3_1$ contains both $\eta$ and $\gamma$ in its interior. The linking numbers between lifts of $\eta$ and $\gamma$ to $\Sigma_q(K)$ are not affected by local knotting of $\xi$. Put differently, if a connected sum $K\#J$ is formed by replacing $\xi$ with a knotted 1-tangle, the linking numbers between lifts of $\eta$ and $\gamma$ to $\Sigma_q(K)$ are equal to the linking numbers between the corresponding lifts to $\Sigma_q(K\#J)$. 
\end{lem}

\begin{proof}
    We first note that the linking numbers of both $\eta$ and $\gamma$ with $K$ are not changed by local knotting of $\xi$ inside $B^3_2,$ so the preimage of $\eta$ (resp. $\gamma$) has the same number of components in $\Sigma_q(K)$ as in $\Sigma_q(K\#J)$ and there is a natural identification between the components in the two covers.

    The $q-$fold cyclic cover of ${B}^3_2$ branched along the trivial 1-tangle $\xi$ is also a 3-ball, denoted $\widehat{B}^3_2$.  The $q-$fold cyclic cover $\Sigma_q(K\#J)$ can be obtained from $\Sigma_q(K)$ by replacing the interior of  $\widehat{B}^3_2$ with a punctured copy of $\Sigma_q(J)$. 
    
    Now let $\eta^k, \gamma^j$ be two connected components of the pre-images of $\eta$ and $\gamma$, respectively, in $\Sigma_q(K)$. Since $\eta, \gamma$ are disjoint from ${B}^3_2$, we know that $(\eta^k \cup \gamma^j)\cap \widehat{B}^3_2 =\emptyset$.   
By a Meyer-Vietoris argument,  $\eta^k$ or $\gamma^j$ are rationally nullhomologous in $\Sigma_q(K)$ if and only if they are so in $\Sigma_q(K\#J)$. Thus, the linking number is defined in $\Sigma_q(K)$ if and only if it's defined in $\Sigma_q(K\#J)$. Assume that this is the case. Then, some integer multiple of $\eta^k$ bounds a 2-chain $N$ in $\Sigma_q(K)$. This 2-chain can be assumed disjoint from $\widehat{B}^3_2$ and therefore can be used to compute the linking number between $\eta^k$ and $\gamma^j$ in $\Sigma_q(K\#J)$.
\end{proof}

\begin{cor} \label{cor:obstruction-destruction}  Let $K\subset S^3$ be a knot and $q>2$ a prime. $K$ admits an embedding in $S^1\times D^2$ of winding number $q$ with the property that, for the satellite pattern $P$ determined by this embedding, the linking numbers between the lifts of $\eta$ to $\Sigma_q(K)$ are all 0. In particular, the obstruction from~\cite{lidman2022linking} can not be used to show that the pattern  $P$ does not induce a homomorphism on smooth concordance.  
\end{cor}

Of course, if $K$ is not slice, no satellite operation with $P(U)=K$ will induce a homomorphism on concordance.

\begin{proof}
    Let $P'$ a pattern be an alternating cable of winding number $q$. 
    Since $P'$ is isotopic in $S^1\times D^2$ to its mirror, the linking number obstruction vanishes for $P'$. By Lemma~\ref{lem:local-knotting}, introducing $K$ as a local knot in $P'(U)$ does not affect the value of the obstruction. This produces the desired embedding of $K$ into the solid torus.
\end{proof}

On the other hand, the computations we did on the side, only a small number of which are included here, convinced us that that the homomorphism obstruction rarely vanishes.  Next, we give one family of (slice) examples for which the relevant linking numbers take non-zero values.  

\begin{ex}\label{ex:steve} We consider several embeddings, of increasing winding number, of the Stevedore knot into the solid torus, and we show that the satellite operations determined by the resulting patterns do not induce homomorphisms on concordance. The linking numbers between the lifts of $\eta$ are given in Table~\ref{tab:6_1}. We remark that the lifts of $\eta$ in this example are not nullhomologous over $\mathbb{Z}$ so we appeal to Theorem~\ref{thm:obstruction-2}.  
We compute the linking numbers between lifts of $\eta$ in the 2, 3, 4 and 5- fold covers of $K$, just to illustrate that it is easy, even though linking numbers in the 2-fold covers suffice to obstruct the pattern from inducing a homomorphism on concordance. Note that, in each case, an odd multiple of each $\eta_i$ bounds, so its order in homology is odd, as required in the hypotheses of Theorem~\ref{thm:obstruction-2}. Consequently, the order of the lift of each $\eta_i$ is odd as well, so we can apply the Theorem.

\begin{figure}
    \centering
    \includegraphics[width=4in]{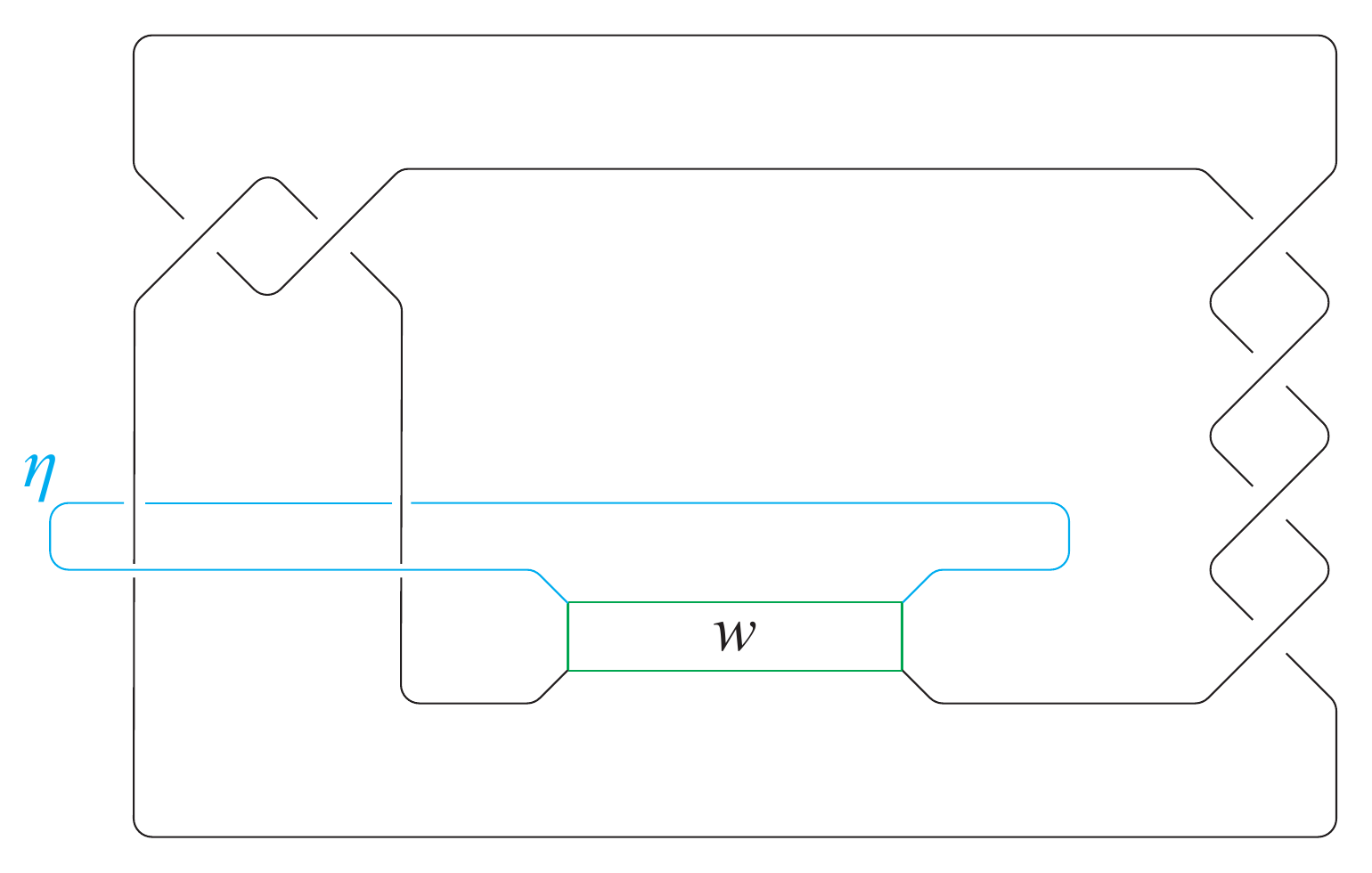}
    \caption{A family of patterns with increasing winding number, each with underlying knot the Stevedore knot. 
 The box contains $w$ full twists, ensuring that the winding number of the pattern in the solid torus is $w$.}
    \label{fig:stevedorefamily}
\end{figure}

\begin{table}
    \centering
    \begin{tabular}{|c|c|c|c|}
    \hline
         Degree  &Winding & Multiple of $\eta_1$& $\text{lk}(\eta^1,\eta^2),\text{lk}(\eta^1,\eta^3),\dots,\text{lk}(\eta^1,\eta^q)$ \\
         of cover $q$& number $w$& which bounds in $\Sigma_q(K)$&\\
         \hline
        $2$&$0$&$9$&$2/9$\\
         \hline
         $2$ & $2$ &$9$& $-7/9$\\
         \hline
         $2$ & $4$ &$9$& $-16/9$\\
         \hline
         $2$ & $6$ &$9$& $-25/9$\\
         \hline
         $2$ & $8$ &$9$& $-34/9$\\
        \hline\hline
        $3$ & $0$ &$7$& $1/7$, $1/7$\\
         \hline
         $3$ & $3$ &$7$& $-6/7$, $-6/7$\\
         \hline
         $3$ & $6$ &$7$& $-13/7$, $-13/7$\\
         \hline
         $3$ & $9$ &$7$ &$-20/7$, $-20/7$\\
         \hline 
         $3$ & $12$ &$7$ &$-27/7$, $-27/7$\\
         \hline\hline
         $4$ & $0$ &$45$& $1/9$, $4/45$,  $1/9$\\
         \hline
         $4$ & $4$ &$45$& $-8/9$, $-41/45$,   $-8/9$\\
         \hline
         $4$ & $8$ &$45$& $-17/9$, $-86/45$, $-17/9$\\
         \hline
         $4$ & $12$ &$45$& $-26/9$, $-131/45$,   $-26/9$\\
         \hline 
         $4$ & $16$ &$45$& $-35/9$, $-176/45$,   $-35/9$\\
         \hline\hline
          $5$ & $0$ &$31$& $3/31$, $2/31$, $2/31$, $3/31$\\
         \hline
         $5$ & $5$ &$31$& $-28/31$, $-29/31$, $-29/31$, $-28/31$\\
         \hline
         $5$ & $10$ &$31$& $-59/31$, $-60/31$, $-60/31$, $-59/31$\\
         \hline
         $5$ & $15$ &$31$& $-90/31$, $-91/31$, $-91/31$, $-90/31$\\
         \hline 
         $5$ & $20$ &$31$& $-121/31$, $-122/31$, $-122/31$, $-121/31$\\
         \hline
         
    \end{tabular}
    \caption{Linking numbers for Example~\ref{ex:steve}:  Stevedore patterns with increasing winding number $w$. See Figure~\ref{fig:stevedorefamily}.}
    \label{tab:6_1}
\end{table}
    
\end{ex}

To conclude, we compute the obstruction in a couple of cases where, in our estimate, the classical methods of~\cite{akbulut1980branched} would be exceedingly challenging to use in practice, at least by hand (and we are not aware of any implementation of the technique). The patterns here are determined by two-bridge slice knots whose Seifert genus grows. For these examples, we also perform the computation for  high degree covers, i.e. for relatively large values of $q$.

\begin{figure}
    \centering
    \includegraphics[width=\textwidth]{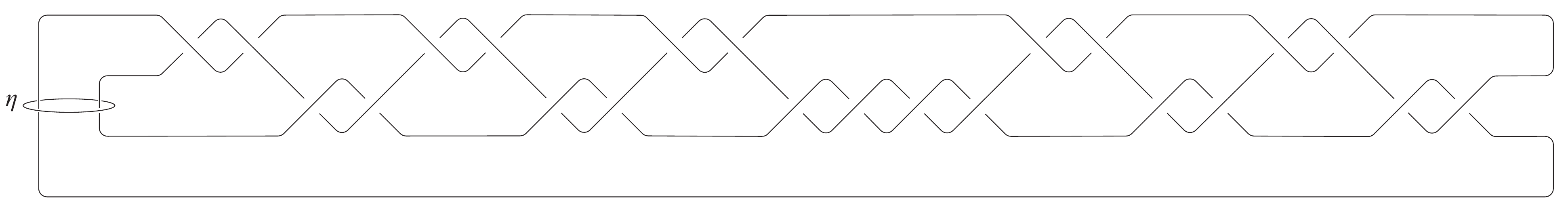}
    \caption{The pattern $P_2$ in the family of slice patterns $P_m$ with winding number 0.}
    \label{fig:2bridgeslicefamily}
\end{figure}

\begin{ex}\label{ex:gnarly-slice} The  2-bridge knots $C(a_1,a_2,\dots, a_n, x, x+2, a_n, \dots, a_2, a_1)$ in the notation of~\cite{lamm2022symmetric} are ribbon.  In this example, we set $a_i=2$ for all $i$, $x=2$, and $n=2m$. We obtain a family of satellite patterns $P_m$ by embedding each of these knots knot in the solid torus in the manner indicated in Figure~\ref{fig:2bridgeslicefamily} for the case $m=2$.  Since all twist parameters are even, the Seifert genus of the knot $P_m$ is $m+1$. The linking numbers of the lifts of $\eta$ for the first few values of $m$ are shown in Table~\ref{tab:2bridgeslice}. Note that, while we don't compute the order in homology of the curves $\eta_i$ precisesly, we verify that this order is odd as it divides an odd number in each case. Based on our computations, a sampling of which is included in Table~\ref{tab:2bridgeslice}, we would be surprised if linking numbers failed to obstruct any of the patterns $P_m$ from inducing homomorphisms on concordance, using Theorem~\ref{thm:obstruction-2}. 
\end{ex}

 We invite the reader to test any pattern they are interested in, using the code provided in \cite{cahn2023cyclic} and the instructions given in the Appendix.

\begin{table}
    \centering
    \begin{tabular}{|c|c|c|c|c|}
    \hline
        $m$ & Seifert & Degree   & Multiple of $\eta_1$  &$\text{lk}(\eta^1,\eta^2),\text{lk}(\eta^1,\eta^3),\dots,\text{lk}(\eta^1,\eta^q)$  \\
        &genus&of cover $q$& which bounds in in $H_1(\Sigma_q)$&\\
        \hline
         0& 1 & 2&$9$&$4/9$\\
         \hline
         $0$&1&$3$&$7$& $2/7$, $2/7$\\
         \hline
         $0$&1&$4$&$45$& $2/9$, $8/45$,  $2/9$\\
         \hline 
         $0$&1&$5$&$31$&$6/31$, $4/31$, $4/31$, $6/31$\\
         \hline\hline
         $1$&3 & $2$&$289$& $120/289$\\
         \hline
         $1$&3&$3$&$151$& $40/151$, $40/151$\\
         \hline
         $1$&3&$4$&$1875$& $60/289$,$3036/18785$, $60/289$\\
         \hline 
         $1$&3&$5$&$7481$& $1350/7481$, $896/7481$,  $896/7481$, $1350/7481$\\
         \hline\hline
         $2$& 5& $2$&$9801$& $4060/9801$\\
         \hline
         $2$&5&$3$&$3457$& $912/3457$, $912/3457$\\
         \hline
         $2$&5&$4$&$8830701$& $2030/9801$, $1405300/8830701$,       $2030/9801$\\
         \hline 
         $2$&5&$5$&$2042281$& $369248/2042281$, $239272/2042281$, \\
         &&&& $239272/2042281,$ $369248/2042281$\\
         \hline
    \end{tabular}
    \caption{Linking numbers for Example~\ref{ex:gnarly-slice}:  patterns with winding number 0 determined by the two-bridge slice knots $P_m$ in Figure~\ref{fig:2bridgeslicefamily}.}
    \label{tab:2bridgeslice}
\end{table}

\section{Acknowledgments} This work was partially supported by NSF grants DMS-2145384 to PC and DMS-2204349 to AK. The core of the paper was chewed through during the 2023 Moab Topology Conference. We thank the organizers Nathan Geer, Mark Hughes, Maggie Miller and Matt Young for all the fruit.

\bibliographystyle{plain}
\bibliography{cycliclinking}

\begin{thebibliography}{10}

\bibitem{heddenconjecture}
Final report.
\newblock In {\em Synchronizing Smooth and Topological 4-Manifolds}. BIRS,
  2016.

\bibitem{akbulut1980branched}
Selman Akbulut and Robion Kirby.
\newblock Branched covers of surfaces in 4-manifolds.
\newblock {\em Mathematische Annalen}, 252:111--131, 1980.

\bibitem{birman1980seifert}
Joan~S Birman and Julian Eisner.
\newblock {\em Seifert and Threlfall, A Textbook of Topology}.
\newblock Academic Press, 1980.

\bibitem{cahn2021dihedral}
Patricia Cahn, Elise Catania, Sarangoo Chimgee, Olivia Del~Guercio, and Jack
  Kendrick.
\newblock Dihedral linking invariants.
\newblock {\em arXiv preprint arXiv:2112.14790}, 2021.

\bibitem{cahn2023cyclic}
Patricia Cahn and Alexandra Kjuchukova.
\newblock Linking numbers in cyclic covers.
\newblock {\em https://github.com/patriciacahn/cyclic\_covers\_linking}.

\bibitem{cahn2018computing}
Patricia Cahn and Alexandra Kjuchukova.
\newblock Computing ribbon obstructions for colored knots.
\newblock {\em Fundamenta Mathematicae}, 253(2), 2021.

\bibitem{cahn2021linking}
Patricia Cahn and Alexandra Kjuchukova.
\newblock Linking numbers in three-manifolds.
\newblock {\em Discrete \& Computational Geometry}, 66(2):435--463, 2021.

\bibitem{cappell1975invariants}
Sylvain~E. Cappell and Julius~L. Shaneson.
\newblock Invariants of 3-manifolds.
\newblock {\em Bulletin of the American Mathematical Society}, 81(3):559--562,
  1975.

\bibitem{cha2002signatures}
Jae~Choon Cha and Ki~Hyoung Ko.
\newblock Signatures of links in rational homology spheres.
\newblock {\em Topology}, 41(6):1161--1182, 2002.

\bibitem{geske2021signatures}
Christian Geske, Alexandra Kjuchukova, and Julius~L Shaneson.
\newblock Signatures of topological branched covers.
\newblock {\em International Mathematics Research Notices}, 2021(6):4605--4624,
  2021.

\bibitem{hedden2021satellites}
Matthew Hedden and Juanita Pinz{\'o}n-Caicedo.
\newblock Satellites of infinite rank in the smooth concordance group.
\newblock {\em Inventiones mathematicae}, 225(1):131--157, 2021.

\bibitem{kjuchukova2018dihedral}
Alexandra Kjuchukova.
\newblock Dihedral branched covers of four-manifolds.
\newblock {\em Advances in Mathematics}, 332:1--33, 2018.

\bibitem{lamm2022symmetric}
Christoph Lamm.
\newblock Symmetric union presentations for 2-bridge ribbon knots.
\newblock {\em Journal of Knot Theory and Its Ramifications}, 30(12), 2022.

\bibitem{lidman2022linking}
Tye Lidman, Allison~N Miller, and Juanita Pinz{\'o}n-Caicedo.
\newblock Linking number obstructions to satellite homomorphisms.
\newblock {\em arXiv preprint arXiv:2207.14198}, 2022.

\bibitem{miller2023homomorphism}
Allison Miller.
\newblock Homomorphism obstructions for satellite maps.
\newblock {\em Transactions of the American Mathematical Society, Series B},
  10(08):220--240, 2023.

\bibitem{perko1964invariant}
Kenneth~A Perko.
\newblock {\em An invariant of certain knots}.
\newblock PhD thesis, Department of Mathematics, 1964.

\bibitem{przytycki2004linking}
J{\'o}zef Przytycki and Akira Yasuhara.
\newblock Linking numbers in rational homology 3-spheres, cyclic branched
  covers and infinite cyclic covers.
\newblock {\em Transactions of the American Mathematical Society},
  356(9):3669--3685, 2004.

\bibitem{reidemeister2013knotentheorie}
Kurt Reidemeister.
\newblock {\em Knotentheorie}, volume~1.
\newblock Springer-Verlag, 2013.

\end{thebibliography}

\section{Appendix}
We explain how to use the code in \cite{cahn2023cyclic} to apply Theorems~\ref{thm:mainA} and~\ref{thm:mainB}. See Figure~\ref{fig:screenshot} for a concrete example.

The first step is to fix an oriented diagram of the link $K\cup \eta\cup \gamma$ where $K$ has writhe 0. Fix a point on $K$ and label the corresponding arc by $k_0$. Proceed in the direction of the orientation of $K$ and assign to consecutive arcs labels $k_i$, increasing the subscript each time the knot encounters an overpass. See Figure~\ref{fig:31cablecells} for reference. Proceed similarly with $\eta$ and $\gamma$, assigning labels $\eta_i$ and $\gamma_i$, respectively.

Next, we input the diagram into the computer, which amounts to recording the following information: the number of components; the sign of each crossing; the over-strand of each crossing. For the latter, this means specifying a component of the link\footnote{Note that we have stated our theorems for the case of two pseudo-branch curves, and we focus on applications involving a single pseudo-branch curve, $\eta$. 
However, the implementation of our algorithm allows for computing linking numbers between lifts of more than two pseudo-branch curves.} and an arc on this component.  The signs are recorded as ``1" and ``-1," and the over-strands are recorded as by specifying which link component the over-strand belongs to, and which arc it is of this component.

To record the above information, as well as to perform computations, we use the following functions:

\begin{enumerate}
    \item $\text{Link}\_\text{component}(~,~ ,~)$ introduces a component to the link diagram. It takes as input: an integer $c$, the number of components in the diagram; a list of 1-s and $-1$-s, the signs of crossings along this new component; a list of integer pairs $[a, b]$ where $a\in \{0, 1, \dots, n-1\}$ specifies the component containing the overstrand at each crossing, and $b$ specifies the arc on this component that is the overstrand. Each new component should be named. In Figure~\ref{fig:screenshot}, the names used are ``knot" and ``$\text{pb}\_1$".
    \item $\text{diagram}(c,m)$ takes as input two integers: $c$ is the number of components in the link diagram; $m\in \{0, 1, \dots, m-1\}$ specifies which component is the branch curve.
    \item $\text{diagram.add}\_\text{component}$ introduces a component, among the ones created in $(1)$, to the diagram. 
    \item  $q\in\mathbb{N}$ specifies the degree of the cover.
    \item $\text{cover}=\text{cyclic}\_\text{cover}(q, \text{diagram})$ generates the combinatorial data encoding the cell structure described in Section~\ref{sec:cellular}.
    \item $\text{cover}.\text{two}\_\text{chain}(a,b)$ prints the coefficients of 2-cells the 2-chain found in Theorem~\ref{thm:mainA}, whose boundary is the $b$-th lift of the $a$-th component of the link diagram. The two cells paired with these coefficients are ordered as follows:

    $$(x_0^1,\dots,x_0^q, x_1^1,\dots, x_1^q,\dots,x_{n-1}^1,\dots, x_{n-1}^q)$$
    \item $\text{cover}.\text{linking}\_\text{number}(a,i,b,j)$ prints the linking number between the $i$-th lift of the $a$-th component and the $j$-th lift of the $b$-th component of the link.
\end{enumerate}

\begin{figure}
    \centering
    \includegraphics[width=7in]{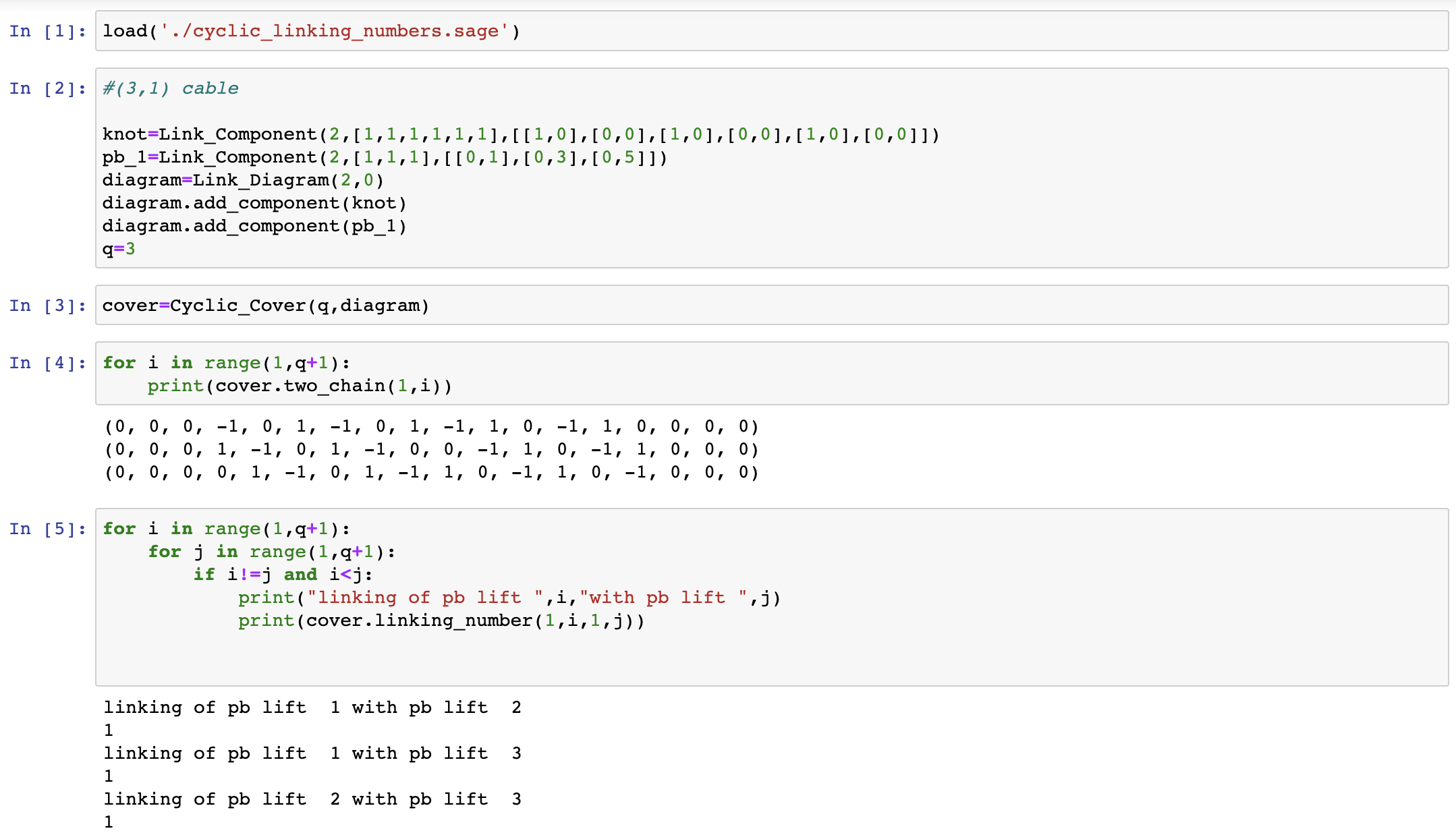}
    \caption{Computing the linking numbers between lifts of $\eta$ for the (3, 1) cable.}
    \label{fig:screenshot}
\end{figure}

\end{document}